\documentclass[11pt, reqno]{amsart}
\usepackage{a4wide}
\usepackage{hyperref}
\usepackage{amsmath}
\usepackage{paralist}
\usepackage{amssymb}
\usepackage{amsthm}
\usepackage{amscd}
\usepackage{graphicx,mathrsfs}
\usepackage{fontenc}
\usepackage{enumitem}
\renewcommand{\subjclassname}{2020 Mathematics Subject Classification}
\begin{document}
\title[Cone degenerate Laplace equation]
{Existence results for nonlinear cone degenerate Laplace equations}
\author[ 
Hua Chen, Xiaochun Liu, Yawei Wei,  Mengnan Zhang]
{
Hua Chen, Xiaochun Liu, Yawei Wei,    Mengnan Zhang}

%
%
\address{Hua Chen  \newline
School of Mathematics and Statistics\\ Wuhan University\\ Wuhan 430072, China}
\email{chenhua@whu.edu.cn}

\address{Xiaochun Liu \newline
School of Mathematics and Statistics\\  Wuhan University\\ Wuhan 430072, China}
\email{xcliu@whu.edu.cn}

\address{Yawei Wei \newline
School of Mathematical Sciences and LPMC\\ Nankai University\\ Tianjin 300071, China}
\email{weiyawei@nankai.edu.cn}


\address{Mengnan Zhang \newline
School of Mathematical Sciences \\ Nankai University\\ Tianjin 300071, China}
\email{1120220030@mail.nankai.edu.cn}

\thanks{Acknowledgements: This work is supported by the NSFC under the grants 12271269, 12131017, 12221001, and  the Fundamental Research Funds for the Central Universities.}



\begin{abstract}
This paper concerns a class of non-divergence  nonlinear elliptic equations driven by the cone degenerate Laplacian, which is motivated by cone calculus. We establish the existence of
viscosity solutions by proving the Alexandrov-Bakelman-Pucci and H\"older estimates.  Furthermore, we obtain the existence of weak solutions by  proving  the  equivalence between  weak solutions and viscosity solutions.
\end{abstract}

\maketitle
\numberwithin{equation}{section}
\newtheorem{Theorem}{Theorem}[section]
\newtheorem{Remark}{Remark}[section]
\newtheorem{lemma}{Lemma}[section]
 \newtheorem{corollary}{Corollary}[section]
\newtheorem{assumption}{Assumption}[section]
\newtheorem{definition}{Definition}[section]
\newtheorem{Proposition}{Proposition}[section]
\newcommand{\R}{\mathbb{R}}
\newcommand{\B}{\mathbf{B}}
\newcommand{\e}{{\mathcal {E}}}
\newcommand{\M}{{\mathcal{H}}}
\newcommand{\Y}{{\mathbf{P}}}

\renewcommand{\subjclassname}
\allowdisplaybreaks

\section{Introduction}

In this paper, we study the non-divergence  nonlinear cone degenerate Laplace equation:%
\begin{equation}\label{U:eq:11}
t^{-2}\triangle_{\mathbb{B}}u
     +t^{-2}(n-2)(t\partial_t u)+h(u)=f(t,x) \ \ \ \ \text{for}\ (t,x) \in \mathbb{B},
\end{equation}
and its Dirichlet problem
\begin{equation}\label{U:H8}
    \begin{cases}
F((t,x),\nabla_\mathbb{B}u,\nabla_\mathbb{B}^2u)+h(u)=f(t,x), \ \ \ & (t,x)\in \mathbb B,\\
u=0,  \ & (t,x)\in \partial{\mathbb B},
\end{cases}
\end{equation}
for  $n\geq 3$. Here
\begin{equation}\label{U:eq:76}
\begin{aligned}
 F((t,x),\nabla_\mathbb{B}u,\nabla_\mathbb{B}^2u)
 :=t^{-2}\triangle_{\mathbb{B}}u
     +t^{-2}(n-2)(t\partial_t u),
 \end{aligned}\end{equation}
The function $f(t,x)$ is continuous in $\mathbb{B}$,  the domain $\mathbb{B}= (0,1)\times X $ is the stretched  conical domain,
  and $X$ is a bounded convex  open set in $\mathbb{R}^{n-1}$ with smooth boundary.  The function  $h: \mathbb{R}\to  \mathbb{R} $ is  continuous, non-increasing, and $h(0)=0$. For $t>0$ and $x\in X$,
 we denote
\begin{equation}\label{U:1.2}
\nabla_\mathbb{B}:=(t \partial_t,\partial_{ x_1},...,\partial _{x_{n-1}})
 \ \text{and} \ \triangle_{\mathbb{B}}:=\text{div}_\mathbb{B}=\nabla_\mathbb{B}\cdot.\end{equation}


The motivation of   equation \eqref{U:eq:11} comes from the calculus on manifolds with conical singularities, see \cite{ESJS,ES,BWS}.  
 A finite dimensional manifold $B$ with conical singularities is a topological space with a finite subset $$\begin{aligned}B_0= \left\{b_1, \dots, b_M\right\} \subset B
 \end{aligned}$$
 of conical singularities, with the  following  two properties:

1. $B \backslash B_0$ is a $C^{\infty}$ manifold;

2. Every $b \in B_0$ has an open neighhourhood $U$ in $B$, such that there is a homeomorphism
$$
\varphi: U \rightarrow X^{\Delta},
$$
and $\varphi$ restricts a diffeomorphism
$$
\varphi^{\prime}: U \backslash\{b\} \rightarrow X^{\wedge} .
$$
Here $X$ is  a bounded open subset in $\mathbb{R}^n$ with smooth boundary, and we set
$$
X^{\Delta}=\overline{\mathbb{R}}_{+} \times X /(\{0\} \times X) .
$$
This local model is interpreted as a cone with the base $X$. Since the analysis is formulated off the singularity, it makes sense to pass to
$$
X^{\wedge}=\mathbb{R}_{+} \times X,
$$
the open stretched cone with the base $X$. In \eqref{U:eq:11} we take the simplest case of finite stretched cone such that
$$
\mathbb{B}=(0,1) \times X \text { and }  \partial\mathbb{B}=(\{1\}\times X)\cup ((0,1]\times (\partial X)).
$$

The Riemannian metric on $X^{\wedge}=\mathbb{R}_{+} \times X$ is endowed with
$$
g:=d t^2+t^2 d x^2.
$$
Then we have
$$
g=\left(g_{i j}\right)=\left(\begin{array}{cccc}
1 & & & \\
& t^2 & & \\
& & \ldots & \\
& & & t^2
\end{array}\right)\quad  \text{and}  \quad\left(g^{i j}\right)= \left(g_{i j}\right)^{-1} =\left(\begin{array}{cccc}
1 & & & \\
& t^{-2} & & \\
& & \ldots & \\
& & &  t^{-2}
\end{array}\right)
$$
 with $G=\sqrt{\left|\operatorname{det}\left(g_{i j}\right)\right|}=t^{n-1}$, where the rest part is $0$. Then the gradient with contra-variant components is as  follows:
$$
\nabla_g u=\left(g^{j 1}\frac{\partial u}{\partial t}+\sum_{k=2}^{n} g^{j k} \frac{\partial u}{\partial x_{k-1}}\right)_{j=1, \ldots, n}=\left(\frac{\partial}{\partial t},t^{-2} \frac{\partial}{\partial {x_1}}, \ldots, t^{-2} \frac{\partial}{\partial{ x_{n-1}}}\right)u.
$$For convenience of notation, we set $$(y_1, y_2, y_3 \cdots, y_n)= (t,x_1, x_2, \cdots, x_{n-1}) .$$
Now we calculate the Laplace operator near to the conical singularity as follows:
$$\begin{aligned}
\Delta_{ g}u&= \operatorname{div}_g\left(\nabla_g u\right)
\\&=   G^{-1} \left( \sum_{j=1}^{n} \frac{\partial}{\partial y_{j}}\right)\left(G\left( \sum_{k=1}^{n} g^{j k} \frac{\partial u}{\partial y_{k}}\right)\right)\\&=:  t^{-2} \operatorname{div}_{\mathbb{B}}\left( \nabla_{\mathbb{B}} u\right)+t^{-2}(n-2)\left(t\partial _t\right)u,\end{aligned}
$$
 which appears at Example 4 of Chapter 7 in \cite{YVE}.

 Viscosity solutions were introduced by M. G. Crandall and P. L. Lions in \cite{MP} for the case of Hamilton-Jacobi equations and extended to the case of second-order elliptic equations by the authors  in \cite{LMMA} and  \cite{HP2}.

The  Alexandrov-Bakelman-Pucci (ABP) estimate is an important part of the regularity research,  which is commonly used to study the Harnack inequality and H\"older  estimate.  The work in \cite{LX,LMMA} obtained the ABP estimate for fully nonlinear elliptic equations on a bounded domain.  Later,  the  authors in \cite{CI,GPF} extended  the above result to  a class of fully nonlinear equations modeled on the $p$-Laplacian with $p>1$.
 X. Cabr\'{e} in \cite{XC} imposed condition
 $\text{(G)}$ on the domain to ensure  that the ABP estimate is valid in an unbounded domain, which is,

  $\text{(G)}$:
 there exist positive real numbers $\sigma,\tau \in (0,1)$ and   $R_0$ such that for each $y\in \Omega$ there is a ball $B(r_y)$ and $r_y\leq R_0$ satisfying
 $$ y\in B(r_y) \ \text{and}\  |B(r_y)\setminus \Omega_{y,\tau}|\geq \sigma |B(r_y)|,$$
 where $B(r_y)$ is  a ball with radius $r_y$, $\Omega_{y,\tau}$ is the connected component of $\Omega\cap B(r_y/\tau)$  containing $y$ and $|\cdot|$ denotes the $n$-dimensional Lebesgue measure. Later,  the authors  in \cite{VCAV}  introduced a weaker condition  $\text{(wG)}$ form of $\text{(G)}$.

 $\text{(wG)}$: the $r_y$ in  $\text{(G)}$  is not required to be uniformly bounded above,  i.e., $R_0=+\infty$.  \\
  Related work can also be found in \cite{AV3,MLA,IFA}.
   Based on those studies,  the  authors in \cite{IIV} proved  the ABP estimate and further obtained the  H\"older estimate under a slightly stronger condition $(\text{G}^d)$ of  $\text{(G)}$ in which  there are real numbers $a_1>1$ and $a_2>0$ such that $r_y\leq a_1 d(y)$ and $d(y)\leq a_2$. Here $d(y)$ is the distance from point $y$ to $\partial \Omega$.  Finally, we mention that the  ABP estimate has been recently studied in \cite{SKAS,AV,PG} and the references therein.

The existence and uniqueness of viscosity solutions have been widely studied.   The  authors in \cite{G4,MP17,HI20}  studied  the existence of  viscosity solutions by vanishing viscosity method and approximation arguments.  Later,   H.Ishii  in \cite{HI30} made improvements  by  applying Perron's method, which gives the existence and uniqueness by constructing a viscosity supersolution and a  viscosity  subsolution and combining the  comparison principle.  Afterwards,  Perron's  method has been  widely applied in \cite{IF22,IF2,MHP}.  It is worth noting that the construction of viscosity supersolutions and subsolutions usually requires that the domains  have a smooth boundary.  For some domains without a smooth boundary,  the work in \cite{GPA,IF3} established  the existence results by domain approximation and function convergence.
Recently,  many other articles have investigated  the existence of viscosity solutions, such as \cite{JTRBV,AV2,JZ}.


The existence of weak solutions is also an important topic in the field of partial differential  equations on singular manifolds.  For instance,  via  variational    method,  there are a series of existence results  for   different  singular manifolds  such as in  \cite{HXW2,HXW,HTW,HW3,WYW}.  In particular,  for a class of nonlinear elliptic equation on manifolds with conical singularities, the authors in  \cite{HXW2}  established  the corresponding Sobolev inequality and
Poincar\'e inequality on the cone Sobolev spaces, and then, as an application of such inequalities,  they proved  the existence of non-trivial weak solution for Dirichlet boundary value problem.  On the other hand,  the equivalence  between weak solutions and viscosity solutions is also an interesting topic
(see \cite{VP,MO,MMPO}), which provides an idea to study the existence of weak solutions.  More details can be also found in \cite{BBMM,JTE,JS2} and references therein.


Our main work is  as follows. Firstly, we research   the existence of
 viscosity solutions to the problem $\eqref{U:H8}$ by exploring ABP estimate of  viscosity solutions for \eqref{U:eq:11} and  weighted H\"older estimate  for $\eqref{U:H8}$. Secondly, we prove the  the equivalence of weak solutions and viscosity solutions, which implies that the existence of weak solutions to \eqref{U:H8}.

Now, we will present the  notations and  assumptions used in this paper.
\begin{itemize}[leftmargin=*]\label{U:notation}
\item[]$\bullet\ S^n$: the space of $n\times n$ real symmetric matrices.
\item[] $\bullet \ USC$:  upper  semicontinuous.
\item[]$\bullet\  LSC$:  lower  semicontinuous.
\item[]$\bullet\ B((t,x),r)$: the Euclid ball with $(t,x)$ as the center and $r$ as the radius.
\item[]$\bullet\ B_{\mathbb{B}}((t,x),r)$: the cone ``ball'' with $(t,x)$ as the center and $r$ as the radius defined in \eqref{U:eq:73}.
 \end{itemize}

%
%
%
%

  When we don't emphasize $(t,x)$, we simply write $B((t,x),r)$ and  $B_{\mathbb{B}}((t,x),r)$ as $B(r)$ and $B_{\mathbb{B}}(r)$ respectively, for abbreviation.
 \begin{definition}[\cite{CLT}, Distance on cone]\label{U:cone 9}
Since the metric on cone is $ds^2=\frac{1}{t^2}(
dt)^2+\Sigma_{i=1}^{n-1}(dx_i)^2$, we obtain that the distance between points $(t,x)=(t,x_1,\cdot\cdot\cdot,x_{n-1})$ and $(t_0,x_0)=(t_0,x_1^0,\cdot\cdot\cdot,x_{n-1}^0)$ on cone is
 \begin{equation}
  d_{\mathbb{B}}((t,x),(t_0,x_0))=\sqrt{(\ln t-\ln t_0)^2+\Sigma_{i=1}^{n-1}(x_i-x_i^0)^2},
  \end{equation}
and we define $|(t,x)-(t_0,x_0)|_{\mathbb{B}}=: d_{\mathbb{B}}((t,x),(t_0,x_0))$. Furthermore, we define  $ d_{\partial\mathbb{B}}{(t,x)}$  as the distance from $(t,x)$ to $\partial \mathbb{B}$, with the specific form as follows:
 \begin{equation}\label{U:eq:85}
 d_{\partial\mathbb{B}}{(t,x)}=\inf\{ d_{\mathbb{B}}\left((t,x),(s,y)\right)\ \big{|}\ \mbox{for any}\ (s,y)\in \partial \mathbb{B}\}.
\end{equation}

For simplicity, we introduce the open ``ball'' in $ \mathbb{R}_{+}^{n} $ in the sense of measure $\frac{dt}{t}dx$,  with center $w=(s,y)=(s,y_1,\cdots,y_{n-1})\in \mathbb{R}_{+}^{n}$ and radius r,  as follows:
\begin{equation}\label{U:eq:73}
B_{\mathbb{B}}((s,y),r):= \left\{ (t,x) \in \mathbb{R}_+^n\ \big{|}\ (\ln t-\ln s)^2+\Sigma_{i=1}^{n-1}(x_i-y_i)^2<r^2\right\}.
 \end{equation}


\end{definition}

\begin{definition}\label{U:def-1}
For the weight data $\gamma \in \R$, $1\leq q<\infty$ and $(t,x)\in \mathbb{B}$, we say that $u(t,x)\in
\mathbb{L}_q^\gamma(\mathbb{B})$ if $u\in\mathcal{D}^\prime(\mathbb{B})$ and
$$\| u\|_{\mathbb{L}_q^\gamma(\mathbb{B})}=\big(\int_{\mathbb{B}}
|{\left(t(1-t)d(x,\partial X)\right)}^{\frac{n}{q}-\gamma}u(t,x)|^q \frac{dt}{t}dx\big)^{\frac{1}{q}}<
+\infty,$$
where $d(x,\partial X)=\inf\{ \vert x-y\vert \  \big{|}\ \mbox{for any}\ y\in \partial X\}.$
\end{definition}

\begin{assumption}\label{U:Assumption1.1}
    We call that the local stretched cone $\mathbb{B}$ satisfies the condition $(G^d_{\mathbb{B}})$ if
there exist  $K_0\geq 1, \ 0<\sigma<1, \ d_0>0$ such that for any $(t,x)\in \mathbb{B}$,  there exists a n-dimensional cone ``ball'' $ B_{\mathbb{B}}({ \tilde {R}_{(t,x)}} )$ of radius $ \tilde{R}_{(t,x)}\leq K_0d_{\partial\mathbb{B}}{(t,x)}$  and $d_{\partial\mathbb{B}}{(t,x)}\leq d_0$, satisfying
$$(t,x)\in B_{\mathbb{B}}({ \tilde {R}_{(t,x)}} )\ \ \text{and} \ \ \ \vert B_{\mathbb{B}}({ \tilde {R}_{(t,x)}} ) \backslash \mathbb{B} \vert_{\mathbb{B}} \geq \sigma \vert B_{\mathbb{B}}({ \tilde {R}_{(t,x)}} )\vert_{\mathbb{B}}. $$

 Here $|\cdot |_{\mathbb{B}}$ denotes   the volume under cone metric, i.e.,
for any $\Omega \subset \mathbb{R}_{+}^{n}$ , the volume of $\Omega$ under cone metric is
 $|\Omega|_{\mathbb{B}}=\int_{\Omega}\frac{dt}{t}dx.$
 \end{assumption}

The above  Assumption goes back to condition $\text{(G)}$ of \cite{XC},  as well as variants  $\text{(wG)}$ and  $(\text{G}^d)$ contained in  \cite{VCAV} and \cite{IIV} respectively.

 \begin{assumption}\label{U:assumption2.2}
 There exists  a sequence of open convex  domains  $\{H_{j}\}$, where each $H_j$ satisfies $(G^d_{H_{j}})$ and has a $C^2$ boundary,  such that
\begin{equation}\label{U:eq:74} H_{j}\Subset H_{j+1}\Subset \mathbb{B} \ \ \ \text{and } \ \ \ \cup_{j}H_{j} =\mathbb{B}.\end{equation}
 \end{assumption}

 Clearly, Assumption \ref{U:Assumption1.1} and  Assumption \ref{U:assumption2.2} hold.

In fact, since $X$ is convex, $\mathbb{B}$ is convex. For all $(t,x)\in \mathbb{B}$, we   select $B_{\mathbb{B}}({ \tilde {R}_{(t,x)}} )$ with   radius $\tilde{R}_{(t,x)}=2d_{\partial\mathbb{B}}{(t,x)}\leq 2 diam(X) $  and center at   $\partial \mathbb{B}$,  then
 $$(t,x)\in B_{\mathbb{B}}({ \tilde {R}_{(t,x)}} )\ \ \text{and} \ \ \ \vert B_{\mathbb{B}}({ \tilde {R}_{(t,x)}} )  \backslash \mathbb{B} \vert_{\mathbb{B}} \geq \frac{1}{2} \vert B_{\mathbb{B}}({ \tilde {R}_{(t,x)}} )\vert_{\mathbb{B}},$$
 which implies  Assumption \ref{U:Assumption1.1} is valid with $K_0=2, \sigma=\frac{1}{2}, d_0=diam(X)$.

Since $\mathbb{B}$ is convex,  there exists a sequence of  open convex domains  $\{H_{j}\}$, where each $H_j$   has a $C^2$ boundary and satisfies \eqref{U:eq:74}.  Furthermore, according to the convexity of  $\{H_{j}\}$,  we have that $\{H_{j}\}$ satisfies $(G^d_{H_{j}})$  with $K_0=2, \sigma=\frac{1}{2}, d_0=diam(X)$. So Assumption \ref{U:assumption2.2}  is also valid.

 Now, we state our main results in the following.

\begin{Theorem}[Alexandrov-Bakelman-Pucci estimate]\label{U:T3}
    Let $v\in USC(\mathbb {\overline{B}})$ be a  viscosity subsolution of \eqref{U:eq:11}, and $\mathop{\sup  }\limits_{\mathbb{{B}}}v ^{+}<+ \infty$. If $t^2f\in C(\overline{\mathbb{B}}) \cap \mathbb{L}^{1}_{n}(\mathbb{B})$,  
     then
      \begin{equation}\label{U:H11}
          \mathop{\sup }\limits_{\mathbb{B}}v^+ \leq\mathop{\sup }\limits_{\partial\mathbb{B}}v^++C^*K_0d_0 {\vert\vert t^2f^-\vert\vert_{ \mathbb{L}^{1}_{n}\left(\mathbb{B} \right)}},
      \end{equation}
      where the constants $K_0$, $d_0$, $\sigma$  come from \textup{Assumption \ref{U:Assumption1.1}}, and  $C^* $ depends on $n, K_0,d_0,\sigma$.
 \end{Theorem}
 \begin{corollary}\label{U:Theorem4.2}
     Let $v\in    C(\mathbb{\overline{B}})$ be a viscosity solution of \eqref{U:eq:11} and $\mathop{\sup }\limits_{\mathbb{{B}}}\vert v \vert < +\infty$. If $t^2f\in C(\overline{\mathbb{B}}) \cap \mathbb{L}^{1}_{n}(\mathbb{B})$, then
     \begin{equation}\label{U:1.7}
         \mathop{\sup }\limits_{\mathbb{B}}\vert v \vert\leq \mathop{\sup }\limits_{\partial \mathbb{B}}\vert v \vert+C^*K_0d_0 {\vert\vert t^2f\vert\vert_{ \mathbb{L}^{1}_{n}\left(\mathbb{B}\right)}},
     \end{equation}
      where the constants $K_0$, $d_0$, $\sigma$  come from \textup{Assumption \ref{U:Assumption1.1}}, and  $C^* $ depends on $n,K_0,d_0,\sigma$.
 \end{corollary}

\begin{Theorem} [Weighted H\"older estimate]\label{U:C}
      If $v\in C(\mathbb{\overline{B}})$  is a  viscosity solution of \eqref{U:H8} with $\mathop{\sup }\limits_{\mathbb{{B}}}\vert v \vert < +\infty$, and $t^2f\in C(\overline{\mathbb{B}}) \cap \mathbb{L}^{1}_{n}(\mathbb{B})$,  
     then  there exists a constant $\alpha_1$ depending on $K_0,$ $d_0$, $\sigma$, $n$ such that, for all $\rho \in (0,\alpha_1]$ the following estimate
     \begin{equation}\label{U:T1.10}
        ||v(t,x)||_{\rho,\overline{\mathbb B}}\leq C \left( {\vert\vert t^2f\vert\vert_{ \mathbb{L}^{1}_{n}\left(\mathbb{B}\right)}} +h(-C^*K_0d_0 {\vert\vert t^2f\vert\vert_{ \mathbb{L}^{1}_{n}\left(\mathbb{B}\right)}})\right)\\
\end{equation}
holds,  where the constants $K_0$, $d_0$, $\sigma$  come from \textup{Assumption \ref{U:Assumption1.1}}, $C^*$ comes from \eqref{U:1.7}, and   $C $ depends on $n$,  
      $K_0$,   $d_0$, $\sigma$,  $\rho$.
    Here the H\"older norm is
  \begin{equation}\label{U:eq:55}
    ||v(t,x)||_{\rho,\overline{\mathbb B}}=:||v(t,x)||_{L^{\infty}(\overline{\mathbb B})}+\mathop{\sup}\limits_{\overline{\mathbb{B}}\times \overline{\mathbb{B}}} \frac{|v(t,x)-v(s,y)|}{|(t,x)-(s,y)|_{\mathbb{B}}^{\rho}}.
    \end{equation}
   \end{Theorem}

   \begin{Theorem}[Existence for viscosity solution]\label{U:B}
 If  $t^2f\in C(\overline{\mathbb{B}}) \cap  \mathbb{L}^{1}_{n}(\mathbb{B})$, then
     \eqref{U:H8} has at least one  continuous viscosity solution $v$ satisfying \eqref{U:T1.10}.
    \end{Theorem}

\begin{Proposition}\label{U:Theorem 1.5}
If $u\in L^{\infty}(\mathbb{B})$ is a viscosity supersolution to \eqref{U:eq:11} and $f\in C(\overline{\mathbb{B}}) \cap L^{\infty}(\mathbb{B})$, then  $u$ is a  weak supersolution to \eqref{U:eq:11} and $u\in\mathbb{H}_{2}^{1,\frac{n-2}{2}}(\mathbb{B})$.
\end{Proposition}
\begin{Proposition}\label{U:Theorem 1.6}
If $u\in L^{\infty}(\mathbb{B})$ is a viscosity solution to \eqref{U:eq:11} and $f\in C(\overline{\mathbb{B}}) \cap L^{\infty}(\mathbb{B})$, then  $u$ is a  weak solution to \eqref{U:eq:11} and $u\in\mathbb{H}_{2}^{1,\frac{n-2}{2}}(\mathbb{B})$.
\end{Proposition}
\begin{Theorem}[Existence for weak solution]\label{U:Theorem1.7}
    If  $f\in C(\overline{\mathbb{B}}) \cap L^{\infty}(\mathbb{B})$, then
     \eqref{U:H8} has at least one weak solution $u$ satisfying \eqref{U:T1.10} in $\mathbb{H}_{2}^{1,\frac{n-2}{2}}(\mathbb{B})$.
 \end{Theorem}
\begin{Proposition}\label{U:Theorem1.8}
    For any $\gamma\in \mathbb{R}$, if $u\in \mathbb{H}_{2}^{1,\gamma}(\mathbb{B})\cap C(\mathbb{B})$ is  a weak supersolution of
     \eqref{U:eq:11}, then $u$ is a viscosity supersolution of  \eqref{U:eq:11}.
 \end{Proposition}

\begin{Theorem}\label{U:Theorem1.9}
    For any $\gamma\in \mathbb{R}$,  if $u\in \mathbb{H}_{2}^{1,\gamma}(\mathbb{B})\cap C(\mathbb{B})$ is  a  weak solution of
     \eqref{U:eq:11}, then $u$ is a viscosity solution of  \eqref{U:eq:11}.
 \end{Theorem}


This paper concerns a class of non-divergence  nonlinear elliptic equations driven by the cone degenerate Laplacian which motivated by cone calculus, \cite{BWS,ES}. Here we highlight the contributions of this paper.  Firstly, the cone degenerate Laplacian reflects the geometric characteristics of the conical singularity on the straight cone, and has been rarely studied.  Similar to classical situation, in this paper,  we obtain the existence  of viscosity solutions and the existence of distribution solutions of  the Dirichlet problem  for cone degenerate Laplace equation \eqref{U:H8}.  Secondly, the  existence of lower-order terms  and  nonlinearity of \eqref{U:eq:11} cause  challenges.
 And the requirement that the function
$h$ in \eqref{U:eq:11} satisfies only the non-increasing condition renders equation \eqref{U:eq:11} more general.  To obtain  the  existence of weak solution to \eqref{U:eq:11},  we adopt  a  research method: studying  the relationship between weak solutions and viscosity solutions, as well as the existence of viscosity solutions.  Finally, in order to over the difficulties from cone degeneracy in analysis, we  construct new  test functions and establish a framework of cone degenerate operators. Specifically,  we define the cone distance in Definition  \ref{U:cone 9}, weighted H\"older norm in \eqref{U:eq:55} and the weighted Sobolev space $\mathbb{H}^{m,\gamma}_p(\mathbb{B})$  in Definition \ref{U:Definition2.2}


 This paper is organized as follows. In Section 2, we present some preliminaries.  In Section 3, we prove Theorem \ref{U:T3} and  derive the ABP estimate for  viscosity solutions of \eqref{U:eq:11}.
 In Section 4, we obtain the weighted  H\"older estimate for  viscosity solutions of \eqref{U:H8} in Theorem \ref{U:C}. Furthermore, based on Theorem \ref{U:C}, we prove the existence of  viscosity solutions to \eqref{U:H8} in  Theorem \ref{U:B}. In Section 5, we establish the equivalence of weak solutions and viscosity solutions. Furthermore, we  obtain the  existence of weak solutions to \eqref{U:H8} in Theorem \ref{U:Theorem1.7}.

\section{Preliminaries}
In this section, we will present  some definitions and remarks, such as weighted Sobolev spaces, viscosity  solutions and other auxiliary  knowledge.


\begin{definition}[\cite{CHLWZ}]\label{U:Definition2.2}
For $m\in \mathbb{N}$, and $\gamma\in \mathbb{R}$, the spaces
$$\mathbb{H}^{m,\gamma}_q(\mathbb{B}):=\Big\{u\in\mathcal{D}^\prime(\mathbb{B}):
 (t\partial_{t})^\alpha\partial_{x}^\beta u   \in \mathbb{L}_q^\gamma(\mathbb{B})\Big\},$$ for arbitrary $\alpha\in
\mathbb{N}$, $\beta\in \mathbb{N}^{N-1}$, and $\alpha+|\beta|\leq m$. Here $ \mathbb{L}_q^\gamma(\mathbb{B})$ is defined in Definition   \ref{U:def-1}.
 Denote $\mathbb{H}^{m,\gamma}_{q,0}(\mathbb{B})$ as the subspace of $\mathbb{H}^{m,\gamma}_q(\mathbb{B})$, which is defined as the closure of $C_{0}^{\infty}(\mathbb{B})$ with  respect to the norm $||\cdot||_{\mathbb{H}^{m,\gamma}_q(\mathbb{B})}$.
\end{definition}
\begin{Remark}[\cite{CHLWZ}]
$\mathbb{H}^{m,\gamma}_q(\mathbb{B})$ is a Banach space.
\end{Remark}

Through simple calculations, we rewrite  \eqref{U:eq:11} as
    $$
 \begin{aligned}
 &t^{-2}div_ \mathbb{B}(\nabla_ \mathbb{B}u)
     +t^{-2}(n-2)(t\partial_t u)+h(u)=div_ \mathbb{B}(t^{-2}\nabla_ \mathbb{B}u)
    +nt^{-2}(t\partial_t u)+h(u)=f(t,x),
    \end{aligned}$$
    and then we can define the weak supersolution (resp. subsolution) of \eqref{U:eq:11} as follows.
\begin{definition}[weak solution]\label{U:D2}For any $\gamma\in \mathbb{R}$, a function $ u \in \mathbb{H}^{1,\gamma}_2(\mathbb{B})$ is   called a weak supersolution (resp. subsolution) to \eqref{U:eq:11} if  $h(u)\in \mathbb{L}^{n}_{1,loc}(\mathbb{B})$, and
\begin{equation}\label{U:D2.2}
\int \limits_{\mathbb{B}}t^{-2}\nabla_ \mathbb{B}u\cdot \nabla_ \mathbb{B}\psi \frac{dt}{t}dx\geq (resp.\leq)\int \limits_{\mathbb{B}}(-f+h(u)+ nt^{-2}(t\partial _{t} u) )\psi \frac{dt}{t}dx\end{equation}
for all non-negative $\psi\in \mathcal{C}_{0}^{\infty}(\mathbb{B}).$
\end{definition}
Next, we provide an equivalent definition of weak solutions, and give a brief proof.
\begin{lemma}\label{U:D3}For any $\gamma\in \mathbb{R}$, a function $u\in\mathbb{H}^{1,\gamma}_2(\mathbb{B})$ is   called  a weak supersolution (resp. subsolution) to \eqref{U:eq:11} if  $h(u)\in \mathbb{L}^{n}_{1,loc}(\mathbb{B})$, and
\begin{equation}\label{U:D2.3}
\int \limits_{\mathbb{B}}\nabla_ \mathbb{B}u\cdot \nabla_ \mathbb{B}\psi \frac{dt}{t}dx\geq (resp.\leq)\int \limits_{\mathbb{B}}(-t^2f+t^2h(u)+(n-2)(t\partial _{t} u) )\psi \frac{dt}{t}dx\end{equation}
for all non-negative $\psi\in \mathcal{C}_{0}^{\infty}(\mathbb{B}).$
\end{lemma}

\begin{proof} Assume $u\in\mathbb{H}^{1,\gamma}_2(\mathbb{B}) $  satisfies \eqref{U:D2.2} for all non-negative $\psi\in \mathcal{C}_{0}^{\infty}(\mathbb{B})$. Since $t^{2}\psi\in \mathcal{C}_{0}^{\infty}(\mathbb{B})$ and $t^{2}\psi$ is non-negative, substituting $t^{2}\psi$ into \eqref{U:D2.2} yields that \eqref{U:D2.3} holds. Conversely, if $u\in\mathbb{H}^{1,\gamma}_2(\mathbb{B}) $  satisfies \eqref{U:D2.3} for all non-negative $\psi\in \mathcal{C}_{0}^{\infty}(\mathbb{B}),$ then substituting $t^{-2}\psi\in \mathcal{C}_{0}^{\infty}(\mathbb{B})$ into \eqref{U:D2.3} yields that \eqref{U:D2.2} holds.
\end{proof}

We will use the equivalent definition of weak solutions in  Lemma \ref{U:D3}  in the following context.
\begin{definition}[\cite{CHLWZ}, Viscosity solution]\label{U:D1}
 Suppose $G$  be a real-valued continuous function on $\Gamma:=\mathbb B\times \mathbb{R} \times \mathbb{R}^n \times S^n$
and  {for any} $A\ge B$,
 $$G((t,x),w,P,A)  \ge G((t,x),w,P,B).$$
   A function $u\in LSC(\mathbb{B})\ (\text{resp}.\ u\in USC(\mathbb{B}))$ is called a viscosity supersolution \ $(\text{resp.\ subsolution})$
of the following equation
\begin{equation} \label{U:H1}
    G((t,x),u,\nabla_\mathbb{B}u,\nabla_\mathbb{B}^2u)=0,
\end{equation}
 if for all ${\phi}\in C^2(\mathbb{B})$,  the following inequality holds at each local minimum $($resp.
maximum$)$ point $(t_0,x_0)\in \mathbb{B}$ of $u-{\phi}$
$$G((t_0,x_0),u(t_0,x_0),\nabla_\mathbb{B}{\phi}(t_0,x_0),\nabla_\mathbb{B}^2{\phi}(t_0,x_0)) \le (resp. \geq)\  0.$$
We call  $u$ is a viscosity solution of \eqref{U:H1} when it is both a subsolution and a supersolution.
\end{definition}

\begin{definition}[\cite{LX}, Pucci operators]\label{U:D11}
For $0<\lambda\leq  \Lambda$  and $M\in S^n$, we define pucci   operators as follows:
    \begin{equation}
         \mathcal{M}^+_{\lambda,\Lambda}(M)=\sup \limits_{\lambda I\le A\le \Lambda I} tr(AM)=\Lambda \sum \limits_{e_i>0} e_i+\lambda \sum\limits_{e_i<0} e_i,
         \end{equation}
         \begin{equation}
        \mathcal{M}^-_{\lambda,\Lambda}(M)=\inf \limits_{\lambda I\le A\le \Lambda I} tr(AM)=\lambda \sum \limits_{e_i>0} e_i+\Lambda \sum\limits_{e_i<0} e_i,
        \end{equation}
where $\{e_i\}$ are eigenvalues of $M$.
\end{definition}

\begin{Remark}\label{U:R1}

 According to the Definition $\ref{U:D11}$, we can derive
 $$\mathcal{M}^-_{1,1}(\nabla_\mathbb{B}^2u)=\triangle _\mathbb{B} u =div_{\mathbb{B}}(\nabla_\mathbb{B}u)= tr(\nabla_\mathbb{B}^2u) = \mathcal{M}^+_{1,1}(\nabla_\mathbb{B}^2u).$$
 \end{Remark}

\section{\textbf{Alexandrov-Bakelman-Pucci  estimate }}
In this section, we   provide  the proof of Theorem \ref{U:T3} to establish the Alexandrov-Bakelman-Pucci  estimate for viscosity solutions of \eqref{U:eq:11}.

     If $u$  is a  viscosity supersolution of
     \begin{equation} \label{U:eq:22}t^{-2}div_ \mathbb{B}(\nabla_ \mathbb{B}u)
     +t^{-2}(n-2)(t\partial_t u)=f(t,x)
      \end{equation} in $  \mathbb{B}$,
      let us define
\begin{equation}\label{U:3.1}
u_m=\begin {cases}\min \{u,m\}  &\quad \ (t,x)\in A,\\m  &\quad \ (t,x) \in B_{\mathbb{B}}(\tilde{d}) \backslash A,\end {cases}\end{equation}
where $A=\mathbb B \cap B_{\mathbb{B}}(\tilde{d}) \neq\emptyset$, $B_{\mathbb{B}}(\tilde{d}) \backslash A \neq\emptyset$ and $m=\mathop {\inf}\limits_{\partial A \cap B_{\mathbb{B}}(\tilde{d})}u(t,x)$.
 Then $u_m $ is a viscosity supersolution to the following equation
\begin{equation}\label{U:4.22}
    t^{-2}div_ \mathbb{B}(\nabla_ \mathbb{B}u)
     +t^{-2}(n-2)(t\partial_t u)= f^+(t,x) \chi(A),
\end{equation}
where $f^{+}(t,x)=\max\{f(t,x),0\}$,  $\chi(A)$ is the indicator function on $A$.

 In fact,  we  assume  $u_m-\phi$ attains a  local minimum at
    $(t_0,x_0)\in A $ for all $\phi\in C^2(\mathbb{B}) $. If $u_m(t_0,x_0)=u(t_0,x_0)$, then we have
    $$ (u-\phi)(t,x)\geq (u_m-\phi)(t,x)\geq (u-\phi )(t_0,x_0)\ \ \ \ \forall (t,x)\in U,$$
   where $U$ is a neighborhood of $(t_0,x_0)$, which implies $u-\phi $ attains a  local minimum at  $(t_0,x_0)$.  Since $u$ is the
   the viscosity supersolution of \eqref{U:eq:11}, we have
   $$ \left(t^{-2}div_ \mathbb{B}(\nabla_ \mathbb{B}\phi)
     +t^{-2}(n-2)(t\partial_t \phi)\right)\big|_{t_0,x_0}\leq f^+(t_0,x_0)\leq  f^+(t_0,x_0)\chi(A). $$
     Similarly, if $u_m(t_0,x_0)=m$, we have $m-\phi $ attains a  local minimum at  $(t_0,x_0)$, which  indicates that
     $\nabla_ \mathbb{B}\phi(t_0,x_0)=0,$  then we get
     \begin{equation}\label{U:3.2}
     \left(t^{-2}div_ \mathbb{B}(\nabla_ \mathbb{B}\phi)
     +t^{-2}(n-2)(t\partial_t \phi)\right)\big|_{t_0,x_0}=0\leq f^+(t_0,x_0)\chi(A). \end{equation}

       On the other hand , we  assume  $u_m-\phi$ attains a  local minimum at $(t_0,x_0)\in B_{\mathbb{B}}(\tilde{d})  \backslash A$. If $(t_0,x_0) \notin  \partial A\cap B_{\mathbb{B}}(\tilde{d}) $, then obviously $m-\phi $ attains a  local minimum at  $(t_0,x_0)$, and $\eqref{U:3.2}$ holds.
    If $(t_0,x_0)\in \partial A\cap B_{\mathbb{B}}(\tilde{d}) $, then for all $\phi\in C^2(\mathbb{B}) $, since $u(t_0,x_0)\geq m $,  we also  have $m-\phi $ attains a  local minimum at  $(t_0,x_0)$. Thus \eqref{U:3.2} holds.

So $u_m $ is a viscosity supersolution of $\eqref{U:4.22}$ in $ B_{\mathbb{B}}(\tilde{d}) $.

Now, we present a known result from \cite{MLA}.
\begin{lemma}[\cite{MLA}, Lemma 2.2] \label{U:lemma2.2}
Let $b_0\geq 0$ and $0<\tau<1$. Suppose that $u(t,x)\in LSC(B({{{1}/{\tau}}}))$ is  a viscosity supersolution of
\begin{equation}\mathcal{M}_{\lambda, \Lambda}^{-}(\nabla^2u)-b_0|\nabla u|=g,\end{equation}
 with $g\in C(\overline{B({{1}/{\tau}})})$, and $u\geq 0$ in $B({{{1}/{\tau}}})$. Then
$$\left(\frac{1}{\vert B(1)\vert }\int_{B(1)}u^{p_0}dtdx\right)^{\frac{1}{p_0}}\leq C\left(\mathop{\inf }\limits_{B(1)} u+{\vert\vert g^+\vert\vert_{L^{n}(B({1/\tau}))}}\right) ,$$
where $C$ and $p_0$ are positive numbers, depending on $n, \lambda,  \Lambda, b_0,$ and $\tau$.
\end{lemma}

\begin{lemma} [Weak Harnack  inequality]\label{U:lemma3.5}
If $u$ is a non-negative viscosity supersolution of \eqref{U:eq:22} in $B_{\mathbb{B}}({2r})$ with $r\leq K_0d_0+1$,  then there are positive  constants  $p_0$ and $C$ such that
   \begin{equation}\label{U:eq:91}
   \left(\frac{1}{\vert B_{\mathbb{B}}({r})\vert_{\mathbb{B}} }\int_{B_{\mathbb{B}}({r})}u^{p_0}\frac{dt}{t}dx\right)^{\frac{1}{p_0}}\leq C\left(\mathop{\inf }\limits_{B_{\mathbb{B}}({r})} u+r {\vert\vert t^2f^+\vert\vert_{\mathbb{L}^{1}_{n}\left( B_{\mathbb{B}}({2r}) \right)}}\right),\end{equation}
    with the constants $K_0$, $d_0$  from \textup{Assumption \ref{U:Assumption1.1}}, and  $p_0$, $C$ depending on $n, K_0, d_0.$

\end{lemma}

\begin{proof}

  Assume that $B_{\mathbb{B}}({2r})$ is centered at $(\bar{t},\bar{x}).$  Let
 $u_l(s,y)=u(t,x)=u(T(s,y))$  with $$(t,x)=T(s,y)=\left(s^r \bar{t}^{1-r},\bar{x}+r(y-\bar{x})\right).$$ Similarly, for any $\phi(t,x)\in C^2(B_{\mathbb{B}}({2r}))$,  we set $\phi_{l}(s,y)=\phi(t,x)=\phi(T(s,y))$.
  A simple calculation shows that $$s\partial_s (\phi_l)=s \cdot \partial_t \phi\cdot r s^{r-1} \bar{t}^{1-r}=r (t\partial_t )\phi, \
    \partial_y ( \phi_l)=r \partial_x \phi ,
   $$
$$
(s\partial_s)^{2} (\phi_l)=s\partial_s(r\cdot t\partial_t \phi)=r^2\cdot t\partial_t(t \partial_t\phi)=r^2(t\partial_t)^2\phi, \
\partial^2_y (\phi_l)=r^2 \partial^2_x \phi, $$ $$
\partial_y(s\partial_s(\phi_l))=\partial_y(r\cdot t\partial_t \phi)=r^2\partial_x(t\partial_t \phi),\
s\partial_s(\partial_y  \phi_l)=s\partial_s(r\cdot \partial_x \phi)=r^2 t\partial_t(\partial_x\phi)
.
$$  That is,
  \begin{equation}\label{U:eq:13}\nabla_\mathbb{B}\phi_l(s,y)=r\nabla_\mathbb{B}\phi(t,x), \  \nabla^2_\mathbb{B}\phi_l(s,y)= r^2\nabla^2_\mathbb{B}\phi(t,x).\end{equation}
For any $\phi_{l}\in C^{2}(B_{\mathbb{B}}({2}))$,  if $u_{l}-\phi_{l}$ attains local minimum at $(s_0, y_0)$ ,   then $u-\phi$ attains local minimum at $(t_0, x_0)$ with $(t_0, x_0)=T(s_0, y_0)$.
 Since  $u$ is a viscosity supersolution of \eqref{U:eq:22},  by \eqref{U:eq:13}, we have
 $$\left(tr(\nabla^2_\mathbb{B}\phi_l)
     +(n-2)r(s\partial_s \phi_l)\right)\big|_{(s_0, y_0)} \leq \left(r^2(t^2 f(t,x) )\circ T\right)\big|_{(s_0, y_0)}.$$
      So $u_l$ is the viscosity supersolution of the following equation,
\begin{equation}\label{U:eq:14}
\begin{aligned}
 tr\left(\nabla_\mathbb{B}^2u_l\right)+r(n-2){(s\partial_ s) u_l}
   = r^2{ \left(t^{2}f(t,x)\right)\circ T}.
  \end{aligned}
\end{equation}
We set
$$a=\ln s,\ \  \overset{T}{} u_l(a,y)=u_l(s,y), \overset{T}{} \phi_l(a,y)=\phi_l(s,y)$$ $$\overset{T}{}\left({\left(t^{2}f \right)\circ T}\right)(a,y)=\left({\left(t^{2}f\right)\circ T}\right)(s,y).$$
We obtain through calculation that
$$\partial_a (\overset{T}{} \phi_l)=\partial_s \phi_l\cdot e^a=s\partial_s \phi_l, \
    \partial_y (\overset{T}{} \phi_l)=\partial_y \phi_l,
   $$
$$
\partial^2_a (\overset{T}{} \phi_l)=\partial_a (s\partial_s \phi_l)=s\partial_s(s \partial_s\phi_l)=(s\partial_s)^2\phi_l, \
\partial^2_y (\overset{T}{} \phi_l)=\partial^2_y \phi_l, $$ $$
\partial_y(\partial_a (\overset{T}{} \phi_l))=\partial_y(s\partial_s \phi_l)=\partial_y(s\partial_s \phi_l),\
\partial_a(\partial_y (\overset{T}{} \phi_l))=\partial_a(\partial_y \phi_l)=s\partial_s(\partial_y\phi_l)
.
$$  That is,
\begin{equation}\label{U:eq:866}
 \nabla (\overset{T}{} \phi_l)(a,y)=\nabla_\mathbb{B}\phi_l(s,y),\nabla^2(\overset{T}{} \phi_l)(a,y)=\nabla_\mathbb{B}^2 \phi_l(s,y).\end{equation}
For any $\overset{T}{} \phi_{l}\in C^{2}(B({2}))$,  if $\overset{T}{}u_{l}-\overset{T}{}\phi_{l}$ attains local minimum at $(a_0, y_0)$,   then $u_{l}-\phi_{l}$ attains local minimum at $(s_0, y_0)$ with $a_0=\ln s_0$. Since  $u_{l}$ is a viscosity supersolution of \eqref{U:eq:14},  by \eqref{U:eq:866}, we have
$$
\begin{aligned}
\left( tr\left(\nabla^2 (\overset{T}{} \phi_{l})\right)+r(n-2){\partial _a(\overset{T}{} \phi_{l})}\right)\big|_{(a_0, y_0)}\leq \left( {\overset{T}{}\left( \left(t^{2}f(t,x) \right)\circ T\right)r^2}\right)\big|_{(a_0, y_0)}.
  \end{aligned}
$$So $\overset{T}{} u_l$ is the viscosity supersolution of the following equation,
$$
\begin{aligned}
 tr\left(\nabla^2 (\overset{T}{}  u_l)\right)+r(n-2){\partial _a(\overset{T}{} u_l)}= {\overset{T}{}\left( \left(t^{2}f(t,x) \right)\circ T\right)r^2}.
  \end{aligned}
$$
 By Remark \ref{U:R1}, $ \overset{T}{}u_l$ satisfies
\begin{equation}\label{U:eq:79}
\mathcal{M}_{1,1}^{-}\left( \nabla^2(\overset{T}{} u_l)\right)+r(n-2) \partial_a (\overset{T}{} u_l) \leq {\overset{T}{}\left( \left(t^{2}f(t,x) \right)\circ T\right)r^2}\end{equation} in viscosity sense.
Choosing $b_0=(n-2)(K_0d_0+1)$ and $\tau=1/2$ , by Lemma \ref{U:lemma2.2} we have
$$\left(\frac{1}{\vert B(1)\vert }\int_{B(1)}(\overset{T}{} u_l)^{p_0}dady\right)^{\frac{1}{p_0}}\leq C\left(\mathop{\inf }\limits_{B(1)} \overset{T}{} u_l+ r^2{\vert\vert {\overset{T}{}\left( \left(t^{2}f^+ \right)\circ T\right)}\vert\vert_{L^{n}(B(2))}}\right) ,$$
and then
 $$\left(\frac{1}{\vert B_{\mathbb{B}}({r})\vert_{\mathbb{B}} }\int_{B_{\mathbb{B}}({r})}u^{p_0}\frac{dt}{t}dx\right)^{\frac{1}{p_0}}\leq C\left(\mathop{\inf }\limits_{B_{\mathbb{B}}({r})} u+r {\vert\vert t^2f^+\vert\vert_{\mathbb{L}^{1}_{n}\left( B_{\mathbb{B}}({2r}) \right)}}\right).$$
\end{proof}

\begin{lemma}[Boundary Weak Harnack  inequality]\label{U:lemma3.31}
 If $u_m$ is a non-negative viscosity supersolution of \eqref{U:4.22} in $B_{\mathbb{B}}({2r})$ with $r\leq K_0d_0+1$, and  $A=B_{\mathbb{B}}({2r})\cap \mathbb{B}$,  then there are positive  constants  $p_0$ and $C$ such that
   \begin{equation}\label{U:eq:44}
   \left(\frac{1}{\vert B_{\mathbb{B}}({r})\vert_{\mathbb{B}} }\int_{B_{\mathbb{B}}({r})}u_m^{p_0}\frac{dt}{t}dx\right)^{\frac{1}{p_0}}\leq C\left(\mathop{\inf }\limits_{B_{\mathbb{B}}({r})} u_m+r {\vert\vert t^2f^+\vert\vert_{\mathbb{L}^{1}_{n}\left(A\cap B_{\mathbb{B}}({2r}) \right)}}\right),\end{equation}
    with the constants $K_0$, $d_0$  from \textup{Assumption \ref{U:Assumption1.1}}, and  $p_0$, $C$ depending on $n, K_0, d_0.$
\end{lemma}
\begin{proof} Since $u_m$ is a non-negative viscosity supersolution of \eqref{U:4.22} in $B_{\mathbb{B}}({2r})$,  $u_m$ is a non-negative viscosity supersolution of \eqref{U:eq:22} with $f^+(t,x) \chi(A)$  instead of $f(t,x)$
in $B_{\mathbb{B}}({2r})$.  By  Lemma \ref{U:lemma3.5}, we  easily draw the conclusion.
\end{proof}
\begin{Remark}
 The constant $C$ in \eqref{U:eq:91} and \eqref{U:eq:44} dosen't depend on $r$, that is, $C$ is the same for any $r\leq K_0d_0+1$.

\end{Remark}

Now we can prove  Theorem \ref{U:T3}.

 \begin{proof}[ \textbf{Proof of Theorem \ref{U:T3}}]
     Let $M=\mathop{\sup \limits_{\mathbb{B}} v^+}$, and  $ u=M-v^+\geq0 $. Since $v$ is a viscosity  subsolution of \eqref{U:eq:11}, we claim that $u$ is the viscosity supersolution of
$$t^{-2}div_ \mathbb{B}(\nabla_ \mathbb{B}u)
     +t^{-2}(n-2)(t\partial_t u)=f^-(t,x),\ \ (t,x)\in\mathbb{B}.$$
Indeed, for all $\varphi\in C^{2}(\mathbb{B})$, if $u-\varphi$ attains a local   minimum in $(t_0,x_0)\in \mathbb{B}$, then for any $(t,x)$ in the  neighborhood of $(t_0,x_0)$, we have
$$ (M-v^+-\varphi)(t,x)= (u-\varphi)(t,x)\geq  (u-\varphi)(t_0,x_0)=(M-v^+-\varphi)(t_0,x_0).$$
If $v^+(t_0,x_0)=0$, then $(M-\varphi)(t,x)\geq (M-\varphi)(t_0,x_0)$.  Thus $\nabla_ \mathbb{B}\varphi(t_0,x_0)=0, \nabla^2_ \mathbb{B}\varphi(t_0,x_0)\leq0 $, and
$$\left(t^{-2}div_ \mathbb{B}(\nabla_ \mathbb{B}\varphi)
     +t^{-2}(n-2)(t\partial_t \varphi)\right)\big|_{t_0,x_0}\leq 0\leq f^-(t_0,x_0);$$
     if $v^+(t_0,x_0)=v(t_0,x_0)\geq 0$, then $ (M-v-\varphi)(t,x) \geq  (M-v-\varphi)(t_0,x_0)$, that is, $(v-(-\varphi))(t,x)\leq (v-(-\varphi))(t_0,x_0)$.  Since $v$ is a viscosity  subsolution of \eqref{U:eq:11}, we have
     $$ \left(t^{-2}div_ \mathbb{B}(\nabla_ \mathbb{B}(-\varphi))
     +t^{-2}(n-2)(t\partial_t (-\varphi))\right)\big|_{(t_0,x_0)}+h(v(t_0,x_0))\geq f(t_0,x_0).$$
Since $v(t_0,x_0)\geq 0$ and  $h(0)=0,$  the non-increasing property of $h$ implies that $h(v(t_0,x_0))\leq 0$.  Then

 $$ \left(t^{-2}div_ \mathbb{B}(\nabla_ \mathbb{B}\varphi)
     +t^{-2}(n-2)(t\partial_t \varphi)\right)\big|_{(t_0,x_0)}\leq -f(t_0,x_0)\leq f^{-}(t_0,x_0).$$
So we have proved the claim.

For all fixed $(t,x)\in \mathbb{B}$,  define $u_m$ as in \eqref{U:3.1} with $\tilde{d}={2\tilde{R}_{(t,x)}}$ and
 $ A=\mathbb{B}\cap B_{\mathbb{B}}({2\tilde{R}_{(t,x)}})$, where $\tilde{R}_{(t,x)}$ defined in  Assumption \ref{U:Assumption1.1},
and then we have $\tilde{R}_{(t,x)}\leq K_0d_0$ and
$u_m$ is the viscosity supersolution of \eqref{U:4.22}
in $ B_{\mathbb{B}}({2\tilde{R}_{(t,x)}})$ with $f^-$ instead of $f^+$. For convenience, we define
\begin{equation}\label{U:eq:96} \tilde{f}(t,x)=t^2f(t,x), \ \ \forall (t,x)\in \mathbb{B}.\end{equation}
By Lemma \ref{U:lemma3.31}, we can get that $u_m$ satisfies
$$\left(\frac{1}{\vert  B_{\mathbb{B}}({\tilde{R}_{(t,x)}})\vert_{\mathbb{B}} }\int_{B_{\mathbb{B}}({\tilde{R}_{(t,x)}})}u_m^{p_0}\frac{ds}{s}dy\right)^{\frac{1}{p_0}}\leq C\left(\mathop{\inf u_m}\limits_{B_{\mathbb{B}}({\tilde{R}_{(t,x)}})}+\tilde{R}_{(t,x)} {\vert\vert \tilde{f}^-\vert\vert_{ \mathbb{L}^{1}_{n}\left(A \cap B_{\mathbb{B}}({2\tilde{R}_{(t,x)}})\right)}}\right),$$
with $C$ depends on $ n,\ K_0,\ d_0$. 

Since $\mathop{\inf }\limits_{B_{\mathbb{B}}({\tilde{R}_{(t,x)}})}u_m\leq \mathop{\inf }\limits_{A\cap B_{\mathbb{B}}({\tilde{R}_{(t,x)}})}u_m\leq u(t,x)=M-v^+(t,x)$, we obtain
$$\left(\frac{1}{\vert B_{\mathbb{B}}({\tilde{R}_{(t,x)}}) \vert_{\mathbb{B}} }\int_{B_{\mathbb{B}}({\tilde{R}_{(t,x)}})}u_m^{p_0}\frac{ds}{s}dy\right)^{\frac{1}{p_0}}\leq C \left( M-v^+(t,x) +\tilde{R}_{(t,x)}{\vert\vert \tilde{f}^-\vert\vert_{ \mathbb{L}^{1}_{n}\left(\mathbb{B}\cap B_{\mathbb{B}}({2\tilde{R}_{(t,x)}}) \right)}}\right). $$
According to Assumption \ref{U:Assumption1.1}, we have  $\vert B_{\mathbb{B}}({\tilde{R}_{(t,x)}})  \backslash \mathbb{B} \vert_{\mathbb{B}} \geq \sigma \vert B_{\mathbb{B}}({\tilde{R}_{(t,x)}})\vert_{\mathbb{B}}$, and then
$$
\begin{aligned}
\left(\frac{1}{\vert B_{\mathbb{B}}({\tilde{R}_{(t,x)}})\vert_{\mathbb{B}}}\int_{ B_{\mathbb{B}}({\tilde{R}_{(t,x)}})  }u_m^{p_0}\frac{ds}{s}dy\right)^{\frac{1}{p_0}}&\geq\left (\frac{1}{\vert B_{\mathbb{B}}({\tilde{R}_{(t,x)}})\vert_{\mathbb{B}}}\int_{B_{\mathbb{B}}({\tilde{R}_{(t,x)}}) \backslash \mathbb{B} }u_m^{p_0}\frac{ds}{s}dy\right)^{\frac{1}{p_0}}\\&=m \left(\frac{\vert B_{\mathbb{B}}({\tilde{R}_{(t,x)}}) \backslash \mathbb{B} \vert_{\mathbb{B}}}{\vert B_{\mathbb{B}}({\tilde{R}_{(t,x)}})\vert_{\mathbb{B}}}\right)^{\frac{1}{p_0}}\geq m \sigma ^{\frac{1}{p_0}}.
\end{aligned}$$
 Then
 \begin{equation}\label{U:3.3}m \sigma ^{\frac{1}{p_0}} \leq C\left(M-v^+(t,x)+\tilde{R}_{(t,x)} {\vert\vert \tilde{f}^-\vert\vert_{ \mathbb{L}^{1}_{n}\left(\mathbb{B}\cap B_{\mathbb{B}}({2\tilde{R}_{(t,x)}})\right)}}\right).\end{equation}
Since $m=\mathop{\inf}\limits_{\partial A \cap B_{\mathbb{B}}({2\tilde{R}_{(t,x)}})}u$, we have
$ m=M-\mathop{\sup v^+ }\limits_{\partial{A}\cap B_{\mathbb{B}}({2\tilde{R}_{(t,x)}})} \geq M -\mathop{\sup }\limits_{\partial\mathbb{B}}v^+.$ After simplification, \eqref{U:3.3} can be inferred that
\begin{equation}\label{U:equation4.16}
\begin{aligned}
v^+(t,x)\leq& (1-\frac{\sigma ^{\frac{1}{p_0}}}{C})M+\frac{\sigma ^{\frac{1}{p_0}}}{C}\mathop{\sup }\limits_{\partial\mathbb{B}}v^++\tilde{R}_{(t,x)} {\vert\vert \tilde{f}^-\vert\vert_{ \mathbb{L}^{1}_{n}\left(\mathbb{B}\cap B_{\mathbb{B}}({2\tilde{R}_{(t,x)}})\right)}}\\ \leq &(1-\frac{\sigma ^{\frac{1}{p_0}}}{C})M+\frac{\sigma ^{\frac{1}{p_0}}}{C}\mathop{\sup }\limits_{\partial\mathbb{B}}v^++\mathop {\sup }\limits_{\mathbb{B}}\tilde{R}_{(t,x)} {\vert\vert \tilde{f}^-\vert\vert_{ \mathbb{L}^{1}_{n}\left(\mathbb{B}\cap B_{\mathbb{B}}({2\tilde{R}_{(t,x)}})\right)}}.
\end{aligned}
\end{equation}
Taking the supremum over $\mathbb{B}$ on the left-hand side of \eqref{U:equation4.16},  we obtain
\begin{equation}\label{U:eq:97}
\mathop{\sup }\limits_{\mathbb{B}}v^+ \leq\mathop{\sup }\limits_{\partial\mathbb{B}}v^++C\mathop {\sup }\limits_{\mathbb{B}}\tilde{R}_{(t,x)} {\vert\vert \tilde{f}^-\vert\vert_{ \mathbb{L}^{1}_{n}\left(\mathbb{B}\cap B_{\mathbb{B}}({2\tilde{R}_{(t,x)}})\right)}}.
\end{equation}
According to Assumption \ref{U:Assumption1.1}, we have  $\tilde{R}_{(t,x)}\leq K_0 d_0$, and  then
$$
\mathop{\sup }\limits_{\mathbb{B}}v^+ \leq\mathop{\sup}\limits_{\partial\mathbb{B}} v^++C^*K_0d_0 {\vert\vert t^2f^-\vert\vert_{ \mathbb{L}^{1}_{n}\left(\mathbb{B}\right)}},
$$
where $C^*$ depends on $n,\ \sigma,\ K_0, \ d_0$.
 \end{proof}

 \begin{proof}[ \textbf{Proof of Corollary \ref{U:Theorem4.2}}]
  Let $M=\mathop{\sup \limits_{\mathbb{B}} (-v)^+}$, and  $ u=M-(-v)^+\geq0 $. Since $v$ is a viscosity  supersolution  of \eqref{U:eq:11}, we claim that $u$ is the viscosity supersolution of
$$t^{-2}div_ \mathbb{B}(\nabla_ \mathbb{B}u)
     +t^{-2}(n-2)(t\partial_t u)=f^+(t,x),\ \ (t,x)\in\mathbb{B}.$$
Indeed, for all $\varphi\in C^{2}(\mathbb{B})$, if $u-\varphi$ attains a local   minimum in $(t_0,x_0)\in \mathbb{B}$, then for any $(t,x)$ in the  neighborhood of $(t_0,x_0)$, we have
$$ (M-(-v)^+-\varphi)(t,x)= (u-\varphi)(t,x)\geq  (u-\varphi)(t_0,x_0)=(M-(-v)^+-\varphi)(t_0,x_0).$$
If $(-v)^+(t_0,x_0)=0$, then $(M-\varphi)(t,x)\geq (M-\varphi)(t_0,x_0)$.  Thus $\nabla_ \mathbb{B}\varphi(t_0,x_0)=0, \nabla^2_ \mathbb{B}\varphi(t_0,x_0)\leq0 $, and
$$\left(t^{-2}div_ \mathbb{B}(\nabla_ \mathbb{B}\varphi)
     +t^{-2}(n-2)(t\partial_t \varphi)\right)\big|_{t_0,x_0}\leq 0\leq f^+(t_0,x_0);$$
     if $(-v)^+(t_0,x_0)=(-v)(t_0,x_0)\geq 0$, then $ (M+v-\varphi)(t,x) \geq  (M+v-\varphi)(t_0,x_0)$, that is, $(v-\varphi)(t,x)\geq (v-\varphi)(t_0,x_0)$.  Since $v$ is a viscosity  supersolution of \eqref{U:eq:11}, we have
     $$ \left(t^{-2}div_ \mathbb{B}(\nabla_ \mathbb{B}\varphi)
     +t^{-2}(n-2)(t\partial_t \varphi)\right)\big|_{(t_0,x_0)}+h(v(t_0,x_0))\leq f^{+}(t_0,x_0).$$
Since $v(t_0,x_0)\leq 0,$  and  $h(0)=0,$  the non-increasing property of $h$ implies that $h(v(t_0,x_0))\geq 0$.  Then

 $$ \left(t^{-2}div_ \mathbb{B}(\nabla_ \mathbb{B}\varphi)
     +t^{-2}(n-2)(t\partial_t \varphi)\right)\big|_{(t_0,x_0)}\leq f^{+}(t_0,x_0).$$
So we have proved the claim.  Next, similar to Theorem \ref{U:T3} , we can conclude that
$$
\mathop{\sup }\limits_{\mathbb{B}}(-v)^+\leq \mathop{\sup }\limits_{\partial \mathbb {B}}(-v)^++C\mathop {\sup }\limits_{\mathbb{B}}\tilde{R}_{(t,x)} {\vert\vert \tilde{f}^+\vert\vert_{ \mathbb{L}^{1}_{n}\left(\mathbb{B}\cap B_{\mathbb{B}}({2\tilde{R}_{(t,x)}})\right)}},
$$ where $\tilde{f}$ is defined in \eqref{U:eq:96}. Combining \eqref{U:eq:97} we get
\begin{equation}\label{U:eq:98}\begin{aligned}
\mathop{\sup }\limits_{\mathbb{B}}\vert v \vert&\leq \mathop{\sup }\limits_{\partial \mathbb{B}}\vert v \vert+C\mathop {\sup }\limits_{\mathbb{B}}\tilde{R}_{(t,x)} {\vert\vert \tilde{f}\vert\vert_{ \mathbb{L}^{1}_{n}\left(\mathbb{B}\cap B_{\mathbb{B}}({2\tilde{R}_{(t,x)}})\right)}}\\& \leq  \mathop{\sup }\limits_{\partial \mathbb{B}}\vert v \vert+C^*K_0d_0 {\vert\vert t^2f\vert\vert_{ \mathbb{L}^{1}_{n}\left(\mathbb{B}\right)}}.\end{aligned}
\end{equation}
\end{proof}
\section{\textbf{The existence of viscosity  solutions}}
In this section, we  obtain  the  weighted H\"older estimate for the viscosity solutions of  \eqref{U:H8} in Theorem \ref{U:C}. Furthermore, we  establish the existence of viscosity solutions to \eqref{U:H8} in Theorem \ref{U:B} through function convergence and domain approximation.
\begin{lemma}[\cite{DN}, Lemma 8.23]\label{U:lemma3.7}
Let $\omega$ be a non-decreasing function on an interval $(0, r_0]$ satisfying  the inequality
$$\omega(\tau r)\leq \gamma \omega(r)+\sigma(r), \ \forall \ r\leq r_0,$$
where $\sigma$ is also non-decreasing and $0<\gamma, \tau<1$. Then for any $\mu\in(0,1)$ and $r\leq r_0$, we have
$$ \omega(r)\leq C\left((\frac{r}{r_0})^{\beta}\omega (r_0)+\sigma (r^{\mu }r_0^{1-\mu})\right),$$
where $C=C(\gamma,\tau)$ and $\beta=\beta(\gamma,\tau,\mu)$ are positive constants.
\end{lemma}

\begin{proof}[ \textbf{Proof of Theorem \ref{U:C}}]
For fixed $(t,x)$ as in Theorem \ref{U:T3}, we reset $M=\mathop{\sup }\limits_{\mathbb{B}\cap B_{\mathbb{B}}({2\tilde{R}_{(t,x)}})}v^+$, where $\tilde{R}_{(t,x)}$ is  defined  in Assumption \ref{U:Assumption1.1}. Since $v$ is the viscosity solution of \eqref{U:H8}, similar to \eqref{U:equation4.16}, we find that
 $$v^+(t,x)\leq (1-\frac{\sigma ^{\frac{1}{p_0}}}{C})M+\frac{\sigma ^{\frac{1}{p_0}}}{C}\mathop{\sup }\limits_{\partial\mathbb{B}}v^++\tilde{R}_{(t,x)} {\vert\vert \tilde{f}^-\vert\vert_{\mathbb{L}^{1}_{n}\left(\mathbb{B}\cap B_{\mathbb{B}}({2\tilde{R}_{(t,x)}})\right)}}$$
 in $\mathbb{B}\cap B_{\mathbb{B}}({2\tilde{R}_{(t,x)}})$ with $\tilde{f}$ defined in \eqref{U:eq:96}.
 Since $v\big|_{\partial{\mathbb{B}}}=0$ and $M=\mathop{\sup }\limits_{\mathbb{B}\cap B_{\mathbb{B}}({2\tilde{R}_{(t,x)}})}v^+$,  we get
  $$v^+(t,x)\leq (1-\frac{\sigma ^{\frac{1}{p_0}}}{C})\mathop{\sup }\limits_{\mathbb{B}\cap B_{\mathbb{B}}({2\tilde{R}_{(t,x)}})}v^++\tilde{R}_{(t,x)} {\vert\vert \tilde{f}^-\vert\vert_{\mathbb{L}^{1}_{n}\left(\mathbb{B}\cap B_{\mathbb{B}}({2\tilde{R}_{(t,x)}})\right)}}.$$
  Multiply both sides of the equation by $\frac{1}{d_{\partial\mathbb{B}}^{\alpha}{(t,x)}+\delta}$ for any $\delta>0$, where $d_{\partial\mathbb{B}}(t,x)$ defined in \eqref{U:eq:85},
 and  then we have
   $$
\frac{v^+(t,x)}{d_{\partial\mathbb{B}}^{\alpha}{(t,x)}+\delta}\leq (1-\frac{1}{C}\sigma ^{\frac{1}{p_0}})\frac{\mathop{\sup }\limits_{\mathbb{B}\cap B_{\mathbb{B}}({2\tilde{R}_{(t,x)}})}v^+}{d_{\partial\mathbb{B}}^{\alpha}{(t,x)}+\delta}+\frac{\tilde{R}_{(t,x)} {\vert\vert \tilde{f}^-\vert\vert_{\mathbb{L}^{1}_{n}\left(\mathbb{B}\cap B_{\mathbb{B}}({2\tilde{R}_{(t,x)}})\right)}}}{d_{\partial\mathbb{B}}^{\alpha}{(t,x)}+\delta}
.$$
Since $v=0 \ \text{on} \ \partial{\mathbb{B}}$,   we can find a $(t^*,x^*)\in \mathbb{B}\cap B_{\mathbb{B}}({2\tilde{R}_{(t,x)}})$ such that
$$\frac{v^+(t,x)}{d_{\partial\mathbb{B}}^{\alpha}{(t,x)}+\delta}\leq (1-\frac{1}{C}\sigma ^{\frac{1}{p_0}})\frac{v^+(t^*,x^*)}{d_{\partial\mathbb{B}}^{\alpha}{(t,x)}+\delta}+\frac{\tilde{R}_{(t,x)} {\vert\vert \tilde{f}^-\vert\vert_{\mathbb{L}^{1}_{n}\left(\mathbb{B}\cap B_{\mathbb{B}}({2\tilde{R}_{(t,x)}})\right)}}}{d_{\partial\mathbb{B}}^{\alpha}{(t,x)}+\delta}
.$$
Since $d_{\partial\mathbb{B}}{(t^*,x^*)}\leq 3\tilde{R}_{(t,x)}$, $\tilde{R}_{(t,x)}\leq K_0d_{\partial\mathbb{B}}{(t,x)}$,  and $K_0\geq1$ , we obtain
$$\begin{aligned}
 \frac{v^+(t,x)}{d_{\partial\mathbb{B}}^{\alpha}{(t,x)}+\delta}\leq & (1-\frac{1}{C}\sigma ^{\frac{1}{p_0}})(3{K_0})^{\alpha}\frac{v^+(t^*,x^*)}{d_{\partial\mathbb{B}}^{\alpha}{(t^*,x^*)}+\delta}+\frac{\tilde{R}_{(t,x)} {\vert\vert \tilde{f}^-\vert\vert_{\mathbb{L}^{1}_{n}\left(\mathbb{B}\cap B_{\mathbb{B}}({2\tilde{R}_{(t,x)}})\right)}}}{d_{\partial\mathbb{B}}^{\alpha}{(t,x)}+\delta}
 \\ \leq &  (1-\frac{1}{C}\sigma ^{\frac{1}{p_0}})(3{K_0})^{\alpha} \mathop{\sup}\limits_{\mathbb{B}}\frac{v^+(t,x)}{d_{\partial\mathbb{B}}^{\alpha}{(t,x)}+\delta}+\mathop{\sup}\limits_{\mathbb{B}}\frac{\tilde{R}_{(t,x)} {\vert\vert \tilde{f}^-\vert\vert_{\mathbb{L}^{1}_{n}\left(\mathbb{B}\cap B_{\mathbb{B}}({2\tilde{R}_{(t,x)}})\right)}}}{d_{\partial\mathbb{B}}^{\alpha}{(t,x)}+\delta}.
 \end{aligned}
$$

Take the supremum over $\mathbb{B}$ on the left-hand side of the above equation, and  then we get
$$\mathop{\sup}\limits_{\mathbb{B}}\frac{v^+(t,x)}{d_{\partial\mathbb{B}}^{\alpha}{(t,x)}+\delta}  \leq  (1-\frac{1}{C}\sigma ^{\frac{1}{p_0}})(3{K_0})^{\alpha} \mathop{\sup}\limits_{\mathbb{B}}\frac{v^+(t,x)}{d_{\partial\mathbb{B}}^{\alpha}{(t,x)}+\delta}  +\mathop{\sup}\limits_{\mathbb{B}}\frac{\tilde{R}_{(t,x)} {\vert\vert \tilde{f}^-\vert\vert_{\mathbb{L}^{1}_{n}\left(\mathbb{B}\cap B_{\mathbb{B}}({2\tilde{R}_{(t,x)}})\right)}}}{d_{\partial\mathbb{B}}^{\alpha}{(t,x)}+\delta}  .
 $$

  For all $\alpha$ such that $(1-\frac{1}{C}\sigma ^{\frac{1}{p_0}})(3{K_0})^{\alpha}< 1$ and $\alpha\leq 1$, we have
  $$\mathop{\sup}\limits_{\mathbb{B}}\frac{v^+(t,x)}{d_{\partial\mathbb{B}}^{\alpha}{(t,x)}+\delta}  \leq C'K_0\mathop{\sup}\limits_{\mathbb{B}}d_{\partial\mathbb{B}}^{1-\alpha} (t,x) {\vert\vert \tilde{f}^-\vert\vert_{\mathbb{L}^{1}_{n}\left(\mathbb{B}\cap B_{\mathbb{B}}({2\tilde{R}_{(t,x)}})\right)}}.$$
   As $\delta\to 0$,  we have
    $$
\mathop{\sup}\limits_{\mathbb{B}}\frac{v^+(t,x)}{d_{\partial\mathbb{B}}^{\alpha}{(t,x)}}  \leq C'K_0\mathop{\sup}\limits_{\mathbb{B}}d_{\partial\mathbb{B}}^{1-\alpha} (t,x) {\vert\vert \tilde{f}^-\vert\vert_{\mathbb{L}^{1}_{n}\left(\mathbb{B}\cap B_{\mathbb{B}}({2\tilde{R}_{(t,x)}})\right)}},$$
where $C'$ depends on $n,\ \sigma,\ K_0,\ d_0,\ \alpha$.
    Similarly, as in the proof of Corollary \ref{U:Theorem4.2}, we perform a similar operation on $-v$ and obtain
$$
\mathop{\sup}\limits_{\mathbb{B}}\frac{v^-(t,x)}{d_{\partial\mathbb{B}}^{\alpha}{(t,x)}}\leq C'K_0\mathop{\sup}\limits_{\mathbb{B}}d_{\partial\mathbb{B}}^{1-\alpha} (t,x) {\vert\vert \tilde{f}^+\vert\vert_{\mathbb{L}^{1}_{n}\left(\mathbb{B}\cap B_{\mathbb{B}}({2\tilde{R}_{(t,x)}})\right)}},
$$
that is,
\begin{equation}\label{U:equation4.19}
\mathop{\sup}\limits_{\mathbb{B}}\frac{|v(t,x)|}{d_{\partial\mathbb{B}}^{\alpha}{(t,x)}}\leq C'K_0\mathop{\sup}\limits_{\mathbb{B}}d_{\partial\mathbb{B}}^{1-\alpha} (t,x) {\vert\vert \tilde{f}\vert\vert_{\mathbb{L}^{1}_{n}\left(\mathbb{B}\cap B_{\mathbb{B}}({2\tilde{R}_{(t,x)}})\right)}}.
\end{equation}

Let $B_{\mathbb{B}}(\tilde{r})$ be a ball of radius $\tilde{r} $ centered at $ (t,x)$. We set
    \begin{equation}
        M_{\tilde{r}}=\mathop{\sup}\limits_{\mathbb B\cap B_{\mathbb{B}}(\tilde{r})}v,\ m_{\tilde{r}}=\mathop{\inf}\limits_{\mathbb B\cap B_{\mathbb{B}}(\tilde{r}) }v,\ \omega({\tilde{r}})=M_{\tilde{r}}-m_{\tilde{r}}.
    \end{equation}
Next, we will consider three different cases.

\noindent$(1)$  If $\tilde{r}\leq 1$, and $\tilde{r}\geq\frac{1}{2}d_{\partial\mathbb{B}}{(t,x)}$, then for $(s,y)\in B_{\mathbb{B}}(\tilde{r})$, we have
    $$ d_{\partial\mathbb{B}}{(s,y)}\leq d_{\partial\mathbb{B}}{(t,x)}+d_{\mathbb{B}}\left( (t,x),(s,y)\right)\leq 3\tilde{r}.$$
    Since $v\in C(\mathbb{B})$, and $v=0 \ \text{on} \ \partial{\mathbb{B}}$,  there exists a $(s_0,y_0)\in \mathbb B\cap B_{\mathbb{B}}(\tilde{r})$ such that
    $$\mathop{\sup}\limits_{\mathbb B\cap B_{\mathbb{B}}(\tilde{r})}|v|= |v(s_0,y_0)|.$$
      Therefore by \eqref{U:equation4.19}, we have
\begin{equation}\label{U:3.136}
    \begin{aligned}
        \omega(\tilde{r})&=\mathop{\sup}\limits_{\mathbb B\cap B_{\mathbb{B}}(\tilde{r}) }v-\mathop{\inf}\limits_{\mathbb B\cap B_{\mathbb{B}}(\tilde{r}) }v\leq 2\mathop{\sup}\limits_{\mathbb B\cap B_{\mathbb{B}}(\tilde{r})}|v|
         \leq 2|v(s_0,y_0)|\\&\leq 2\left(\frac{3\tilde{r} }{d_{\partial\mathbb{B}}{(s_0,y_0)}}\right)^{\alpha}|v(s_0,y_0)|= 2\cdot 3^{\alpha}\tilde{r}^{\alpha}\frac{|v(s_0,y_0)| }{d_{\partial\mathbb{B}}^{\alpha}{(s_0,y_0)}}\\&\leq
       2\cdot 3^{\alpha}\tilde{r}^{\alpha}\mathop{\sup}\limits_{\mathbb B\cap B_{\mathbb{B}}(\tilde{r})}\frac{|v(s,y)| }{d_{\partial\mathbb{B}}^{\alpha}{(s,y)}}
      \leq 2\cdot 3^{\alpha} C'K_0\tilde{r}^{\alpha}\mathop{\sup}\limits_{\mathbb{B}}d_{\partial\mathbb{B}}^{1-\alpha} (t,x) {\vert\vert \tilde{f}\vert\vert_{\mathbb{L}^{1}_{n}\left(\mathbb{B}\cap B_{\mathbb{B}}({2\tilde{R}_{(t,x)}})\right)}}.
        \end{aligned}
    \end{equation}

     \noindent $(2)$ If $\tilde{r}\leq 1$, and $\tilde{r}< \frac{1}{2}d_{\partial\mathbb{B}}{(t,x)} $,
then $B_{\mathbb{B}}(\tilde{r}/2) \subset B_{\mathbb{B}}(\tilde{r})\Subset \mathbb{B}$.
Let \begin{equation}\label{U:eq:118}u_1=v-m_{\tilde{r}}.\end{equation}    For all $\varphi\in C^{2}(\mathbb{B})$, if $u_1-\varphi$ attains a local  minimum in $(t_0,x_0)\in \mathbb{B}$, then  $v-\varphi$ attains a local   minimum $(t_0,x_0)$. Since $v$ is a viscosity solution of \eqref{U:H8}, we have
     $$ \left(t^{-2}div_ \mathbb{B}(\nabla_ \mathbb{B}\varphi)
     +t^{-2}(n-2)(t\partial_t \varphi)\right)\big|_{(t_0,x_0)}+h(v(t_0,x_0))\leq f(t_0,x_0).$$
Since $\mathop{\sup }\limits_{\partial \mathbb{B}}v=0$ and  $h(\cdot)$ is  non-increasing and continuous, by \eqref{U:1.7}, we get $$|h(v)|\leq h(-C^*K_0d_0 {\vert\vert t^2f\vert\vert_{\mathbb{L}^{1}_{n}\left(\mathbb{B}\right)}}),$$ and  then
 $$ \left(t^{-2}div_ \mathbb{B}(\nabla_ \mathbb{B}\varphi)
     +t^{-2}(n-2)(t\partial_t \varphi)\right)\big|_{(t_0,x_0)}\leq f(t_0,x_0)+h(-C^*K_0d_0 {\vert\vert t^2f\vert\vert_{\mathbb{L}^{1}_{n}\left(\mathbb{B}\right)}}).$$
 It implies that $u_1=v-m_{\tilde{r}}$ is a non-negative viscosity supersolution of $$t^{-2}div_ \mathbb{B}(\nabla_ \mathbb{B}u_1)
     +t^{-2}(n-2)(t\partial_t u_1)=f(t,x)+h(-C^*K_0d_0 {\vert\vert t^2f\vert\vert_{\mathbb{L}^{1}_{n}\left(\mathbb{B}\right)}}). $$
Let $u_2=M_{\tilde{r}}-v$. Similarly, we have that $u_2$ is  a non-negative viscosity supersolution of
$$t^{-2}div_ \mathbb{B}(\nabla_ \mathbb{B}u_2)
     +t^{-2}(n-2)(t\partial_t u_2)=-f(t,x)+h(-C^*K_0d_0 {\vert\vert  t^2f\vert\vert_{\mathbb{L}^{1}_{n}\left(\mathbb{B}\right)}}). $$

We apply  $u_1=v-m_{\tilde{r}}$ and  $u_2=M_{\tilde{r}}-v$ to the Lemma \ref{U:lemma3.5} and  get

$$\begin{aligned}
&\left(\frac{1}{\vert B_{\mathbb{B}}({\tilde{r}/2})\vert_{\mathbb{B}} }\int_{B_{\mathbb{B}}({\tilde{r}/2})}(v-m_{\tilde{r}})^{p_0}\frac{dt}{t}dx\right)^{\frac{1}{p_0}} \\ \leq &C\left(\mathop{\inf }\limits_{B_{\mathbb{B}}({\tilde{r}/2})} (v-m_{\tilde{r}})+\frac{\tilde{r}}{2} {\vert\vert t^2f^+\vert\vert_{\mathbb{L}^{1}_{n}\left( B_{\mathbb{B}}({\tilde{r}}) \right)}}+\frac{\tilde{r}^2}{2}h(-C^*K_0d_0 {\vert\vert t^2f\vert\vert_{\mathbb{L}^{1}_{n}\left(\mathbb{B}\right)}})\right).\end{aligned}$$
$$\begin{aligned}
&\left(\frac{1}{\vert B_{\mathbb{B}}({\tilde{r}/2})\vert_{\mathbb{B}} }\int_{B_{\mathbb{B}}({\tilde{r}/2})}(M_{\tilde{r}}-v)^{p_0}\frac{dt}{t}dx\right)^{\frac{1}{p_0}} \\ \leq &C\left(\mathop{\inf }\limits_{B_{\mathbb{B}}({\tilde{r}/2})} (M_{\tilde{r}}-v)+\frac{\tilde{r}}{2} {\vert\vert t^2f^-\vert\vert_{\mathbb{L}^{1}_{n}\left( B_{\mathbb{B}}({\tilde{r}}) \right)}}+\frac{\tilde{r}^2}{2}h(-C^*K_0d_0 {\vert\vert t^2f\vert\vert_{\mathbb{L}^{1}_{n}\left(\mathbb{B}\right)}})\right),\end{aligned}$$

It is known that  $ \vert B_{\mathbb{B}}({\tilde{R}_{(t,x)}})  \backslash \mathbb{B} \vert_{{\mathbb{B}}} \geq \sigma \vert B_{\mathbb{B}}({\tilde{R}_{(t,x)}})\vert_{\mathbb{B}}$ from Assumption \ref{U:Assumption1.1}, and then $2\tilde{R}_{(t,x)}> d_{\partial\mathbb{B}}{(t,x)}$. Also, because $\tilde{r}< \frac{1}{2}d_{\partial\mathbb{B}}{(t,x)}$, we have $\tilde{r}< \tilde{R}_{(t,x)}$. Since $(t,x)$ is the center of  $B_{\mathbb{B}}({\tilde{r}})$, $B_{\mathbb{B}}({\tilde{r}})\subset \mathbb{B}$ and $(t,x)\in B_{\mathbb{B}}({\tilde{R}_{(t,x)}})$, we have $B_{\mathbb{B}}({\tilde{r}})\subset  \mathbb{B}\cap B_{\mathbb{B}}({2\tilde{R}_{(t,x)}})$.
Adding the two equations above gives
$$ M_{\tilde{r}}-m_{\tilde{r}}\leq C\left(M_{\tilde{r}}-M_{\tilde{r}/2}+m_{\tilde{r}/2}-m_{\tilde{r}}+\tilde{r} {\vert\vert \tilde{f}\vert\vert_{\mathbb{L}^{1}_{n}\left(\mathbb{B}\cap B_{\mathbb{B}}({2\tilde{R}_{(t,x)}})\right)}}+{\tilde{r}^2}h(-C^*K_0d_0 {\vert\vert \tilde{f}\vert\vert_{\mathbb{L}^{1}_{n}\left(\mathbb{B}\right)}})\right)$$ with $\tilde{f}$ defined in \eqref{U:eq:96}.
 Transposing terms   gives
\begin{equation}\label{U:3.20}
\begin{aligned}
\omega({\tilde{r}}/ 2) & \leq \frac{C-1}{C} \omega({\tilde{r}})+ {\tilde{r}}^{\alpha}d_{\partial\mathbb{B}}^{1-\alpha}(t,x)  {\vert\vert \tilde{f}\vert\vert_{\mathbb{L}^{1}_{n}\left(\mathbb{B}\cap B_{\mathbb{B}}({2\tilde{R}_{(t,x)}})\right)}} +{\tilde{r}^{\alpha}}h(-C^*K_0d_0 {\vert\vert \tilde{f}\vert\vert_{\mathbb{L}^{1}_{n}\left(\mathbb{B}\right)}})\\
&  \leq \frac{C-1}{C} \omega({\tilde{r}})+ {\tilde{r}}^{\alpha}\mathop{\sup}\limits_{\mathbb{B}} d_{\partial\mathbb{B}}^{1-\alpha}(t,x)  {\vert\vert \tilde{f}\vert\vert_{\mathbb{L}^{1}_{n}\left(\mathbb{B}\cap B_{\mathbb{B}}({2\tilde{R}_{(t,x)}})\right)}}+{\tilde{r}^{\alpha}}h(-C^*K_0d_0 {\vert\vert \tilde{f}\vert\vert_{\mathbb{L}^{1}_{n}\left(\mathbb{B}\right)}}).
\end{aligned}
\end{equation}
From \eqref{U:3.136}, for  $\frac{1}{2}d_{\partial\mathbb{B}}{(t,x)}\leq\tilde{r}\leq 1$,  we have
$$ \begin{aligned}
\omega({\tilde{r}}/ 2) &\leq\omega({\tilde{r}}) \leq \frac{C-1}{C}\omega({\tilde{r}})+\frac{1}{C}\omega({\tilde{r}}) \\
&\leq \frac{C-1}{C}\omega({\tilde{r}})+ \frac{2}{C} \cdot 3^{\alpha} C'K_0\tilde{r}^{\alpha}\mathop{\sup}\limits_{\mathbb{B}}d_{\partial\mathbb{B}}^{1-\alpha} (t,x) {\vert\vert \tilde{f}\vert\vert_{\mathbb{L}^{1}_{n}\left(\mathbb{B}\cap B_{\mathbb{B}}({2\tilde{R}_{(t,x)}})\right)}}.
\end{aligned}$$
Combining   \eqref{U:3.20} and $\alpha\leq 1$, for all $\tilde{r}\leq 1$ we obtain
\begin{equation}\label{U:3.17}
\omega({\tilde{r}}/ 2)  \leq C_1 \omega({\tilde{r}})+C_2 {\tilde{r}}^{\alpha}\left(\mathop{\sup}\limits_{\mathbb{B}}d_{\partial\mathbb{B}}^{1-\alpha} (t,x) {\vert\vert \tilde{f}\vert\vert_{\mathbb{L}^{1}_{n}\left(\mathbb{B}\cap B_{\mathbb{B}}({2\tilde{R}_{(t,x)}})\right)}} +h(-C^*K_0d_0 {\vert\vert \tilde{f}\vert\vert_{\mathbb{L}^{1}_{n}\left(\mathbb{B}\right)}})\right)\end{equation}
with  $C_1=\frac{C-1}{C} $ and  $C_2=\max\{1,(2\cdot 3^{\alpha} C'K_0)/C\}.$

By \eqref{U:equation4.19} we get
$$ \begin{aligned}
\omega{(1)}&\leq\mathop{\sup}\limits_{\mathbb B\cap B_{\mathbb{B}}(1) }v-\mathop{\inf}\limits_{\mathbb B\cap B_{\mathbb{B}}(1) }v\leq 2\mathop{\sup}\limits_{\mathbb B\cap B_{\mathbb{B}}(1)}|v|\leq 2 {d_0}^{\alpha}\mathop{\sup}\limits_{\mathbb B\cap B_{\mathbb{B}}(1)}\frac{|v(s,y)| }{d_{\partial\mathbb{B}}^{\alpha}{(s,y)}}\\&\leq 2C'K_0d_0^{\alpha}\mathop{\sup}\limits_{\mathbb{B}}d_{\partial\mathbb{B}}^{1-\alpha} (t,x) {\vert\vert \tilde{f}\vert\vert_{\mathbb{L}^{1}_{n}\left(\mathbb{B}\cap B_{\mathbb{B}}({2\tilde{R}_{(t,x)}})\right)}}.
\end{aligned}$$
By  \eqref{U:3.17} and   Lemma \ref{U:lemma3.7},   we obtain
\begin{equation}\label{U:3.23}
\begin{aligned}
\omega(\tilde{r}) &\leq C_3\left(\omega(1) \tilde{r}^\beta+ C_2 {\tilde{r}^{\alpha\mu}}\left(\mathop{\sup}\limits_{\mathbb{B}}d_{\partial\mathbb{B}}^{1-\alpha} (t,x) {\vert\vert \tilde{f}\vert\vert_{\mathbb{L}^{1}_{n}\left(\mathbb{B}\cap B_{\mathbb{B}}({2\tilde{R}_{(t,x)}}) \right)}} +h(-C^*K_0d_0 {\vert\vert \tilde{f}\vert\vert_{\mathbb{L}^{1}_{n}\left(\mathbb{B}\right)}})\right)\right)
\\&\leq
C_3\left(\mathop{\sup}\limits_{\mathbb{B}}d_{\partial\mathbb{B}}^{1-\alpha} (t,x) {\vert\vert \tilde{f}\vert\vert_{\mathbb{L}^{1}_{n}\left(\mathbb{B}\cap B_{\mathbb{B}}({2\tilde{R}_{(t,x)}})\right)}}+h(-C^*K_0d_0 {\vert\vert \tilde{f}\vert\vert_{\mathbb{L}^{1}_{n}\left(\mathbb{B}\right)}})\right)\left( 2C'K_0d_0^{\alpha}\tilde{r}^\beta+ C_2 {\tilde{r}^{\alpha\mu}}\right)\\&
\leq C_4 \tilde{r}^{\rho}\left(\mathop{\sup}\limits_{\mathbb{B}}d_{\partial\mathbb{B}}^{1-\alpha} (t,x) {\vert\vert \tilde{f}\vert\vert_{\mathbb{L}^{1}_{n}\left(\mathbb{B}\cap B_{\mathbb{B}}({2\tilde{R}_{(t,x)}})\right)}}+h(-C^*K_0d_0 {\vert\vert \tilde{f}\vert\vert_{\mathbb{L}^{1}_{n}\left(\mathbb{B}\right)}})\right)
\end{aligned}
\end{equation}
for any $\mu \in(0,1)$ and $\rho\leq \min\{{\alpha}\mu,\beta\}$, where $C_3=C_3(C_1)$, $\beta=\beta(C_1, \mu)$ and $C_4=C_3\cdot \max\{2C'K_0d_0^{\alpha},C_2\}$.

\noindent $(3)$ If $\tilde{r}> 1$, similar to \eqref{U:3.136}, by \eqref{U:equation4.19} we have
\begin{equation}\label{U:3.16}
    \begin{aligned}
        \omega(\tilde{r})&=\mathop{\sup}\limits_{\mathbb B\cap B_{\mathbb{B}}(\tilde{r})}v-\mathop{\inf}\limits_{\mathbb B\cap  B_{\mathbb{B}}(\tilde{r}) }v\leq 2\mathop{\sup}\limits_{\mathbb B\cap B_{\mathbb{B}}(\tilde{r})}|v|\leq
       2\cdot {d_0}^{\alpha}\tilde{r}^{\rho}\mathop{\sup}\limits_{\mathbb B\cap  B_{\mathbb{B}}(\tilde{r}) }\frac{|v(s,y)| }{d_{\partial\mathbb{B}}^{\alpha}{(s,y)}}
       \\&\leq 2\cdot C'K_0d_0^{\alpha}\tilde{r}^{\rho}\mathop{\sup}\limits_{\mathbb{B}}d_{\partial\mathbb{B}}^{1-\alpha} (t,x) {\vert\vert \tilde{f}\vert\vert_{\mathbb{L}^{1}_{n}\left(\mathbb{B}\cap B_{\mathbb{B}}({2\tilde{R}_{(t,x)}})\right)}}.
        \end{aligned}
    \end{equation}
 Recalling $\tilde{R}_{(t,x)}\leq K_0d_{\partial\mathbb{B}}{(t,x)}$ and  $d_{\partial\mathbb{B}}(t,x)\leq d_0$, together with  \eqref{U:eq:98}, we  get the interior  H\"older estimate
     \begin{equation}\label{U:1.8}
     ||v(t,x)||_{\rho,\mathbb B}\leq C\left(\mathop{\sup}\limits_{\mathbb{B}}d_{\partial\mathbb{B}}^{1-\alpha} (t,x) {\vert\vert \tilde{f}\vert\vert_{\mathbb{L}^{1}_{n}\left(\mathbb{B}\cap B_{\mathbb{B}}({2\tilde{R}_{(t,x)}})\right)}} +h(-C^*K_0d_0 {\vert\vert \tilde{f}\vert\vert_{\mathbb{L}^{1}_{n}\left(\mathbb{B}\right)}})\right).\end{equation}

Next,  when $(t,x)\in \partial \mathbb{B}$ and $(s,y)\in \partial\mathbb{B}$,
we have $$ \frac{|v(t, x)-v(s, y)|}{|(t, x)-(s, y)|_{\mathbb{B}}^\rho} =0.$$
    Finally, when $(t,x)\in \mathbb{B}$ and $(s,y)\in \partial\mathbb{B}$,
   we can choose $\left(s^{\prime}, y^{\prime}\right) \in \mathbb{B}$ such that $d_{\partial\mathbb{B}}\left(s^{\prime}, y^{\prime}\right) \leq 1$ and $\left(s^{\prime}, y^{\prime}\right)\in [(t,x),(s,y)]$, the line segment between $(t,x)$ and $(s,y)$.
 Hence we have
$$
\begin{aligned}
\frac{|v(t, x)-v(s, y)|}{|(t, x)-(s, y)|_{\mathbb{B}}^\rho} \leqslant \frac{\mid v(t, x)-v\left(s^{\prime}, y^{\prime} \right)\mid}{|(t, x)-(s, y)|_{\mathbb{B}}^\rho}+\frac{\left|v\left(s^{\prime}, y^{\prime}\right)-0\right|}{|(t, x)-(s, y)|_{\mathbb{B}}^\rho} .
\end{aligned}
$$
According  to  the interior  H\"older estimate \eqref{U:1.8}, we get
$$\frac{\mid v(t, x)-v\left(s^{\prime}, y^{\prime}\right) \mid}{|(t, x)-(s, y)|_{\mathbb{B}}^\rho} \leq C\left(\mathop{\sup}\limits_{\mathbb{B}}d_{\partial\mathbb{B}}^{1-\alpha} (t,x) {\vert\vert \tilde{f}\vert\vert_{\mathbb{L}^{1}_{n}\left(\mathbb{B}\cap B_{\mathbb{B}}({2\tilde{R}_{(t,x)}})\right)}} +h(-C^*K_0d_0 {\vert\vert \tilde{f}\vert\vert_{\mathbb{L}^{1}_{n}\left(\mathbb{B}\right)}})\right).$$
Based on \eqref{U:equation4.19}, $\rho \leq \alpha$ and
 $d_{\partial\mathbb{B}}\left(s^{\prime}, y^{\prime}\right)  \leq 1 $, we deduce that $$\begin{aligned}
 \frac{\left|v\left(s^{\prime}, y^{\prime}\right)-0\right|}{|(t, x)-(s, y)|_{\mathbb{B}}^\rho}&\leq
\left|\frac{v\left(s^{\prime}, y^{\prime}\right)}{d_{\partial\mathbb{B}}^\rho\left(s^{\prime}, y^{\prime}\right)}\right|\leq\left|\frac{v\left(s^{\prime}, y^{\prime}\right)}{d_{\partial\mathbb{B}}^\alpha\left(s^{\prime}, y^{\prime}\right)}\right|
\leq C' K_0\mathop{\sup}\limits_{\mathbb{B}}d_{\partial\mathbb{B}}^{1-\alpha} (t,x) {\vert\vert \tilde{f}\vert\vert_{\mathbb{L}^{1}_{n}\left(\mathbb{B}\cap B_{\mathbb{B}}({2\tilde{R}_{(t,x)}})\right)}}.
\end{aligned}
$$

In conclusion, we  get the global H\"older estimate
    $$\begin{aligned}
    ||v(t,x)||_{\rho,\overline{\mathbb B}}&\leq C\left(\mathop{\sup}\limits_{\mathbb{B}}d_{\partial\mathbb{B}}^{1-\alpha} (t,x) {\vert\vert \tilde{f}\vert\vert_{\mathbb{L}^{1}_{n}\left(\mathbb{B}\cap B_{\mathbb{B}}({2\tilde{R}_{(t,x)}})\right)}} +h(-C^*K_0d_0 {\vert\vert \tilde{f}\vert\vert_{\mathbb{L}^{1}_{n}\left(\mathbb{B}\right)}})\right)\\&\leq C\left( {\vert\vert t^2f\vert\vert_{\mathbb{L}^{1}_{n}\left(\mathbb{B}\right)}} +h(-C^*K_0d_0 {\vert\vert t^2f\vert\vert_{\mathbb{L}^{1}_{n}\left(\mathbb{B}\right)}})\right).
    \end{aligned}$$
\end{proof}

Next, we present the existing results in \cite{IF22} and \cite{HP2}.

\begin{lemma}[\cite{IF22}, Proposition 2,3]\label{U:lemma3.8}
Suppose that $f\in L^{\infty}(\Omega)$  and $\tilde{h}: \Omega \times \mathbb{R}\to \mathbb{R}$  such that  $\tilde{h}(x,\cdot)$ is non-increasing and continuous. Then there exists  a viscosity subsolution $u_1$  and  a viscosity supersolution $u_2$ of \begin{equation}\label{U:eq:104}
    \begin{cases}
G(x,\nabla u,\nabla^2u)+\tilde{h}(x,u)=f  \ \ \ & in\  \Omega,\\
u=0 \ & on \ \partial{\Omega}.
\end{cases}
\end{equation}  as follows:
$$u_1=L_{1}\left(\left(1+d(x,\partial \Omega)\right)^{-k_{1}}-1\right), \ u_2=L_{2}\left(1-\left(1+d(x,\partial \Omega)\right)^{-k_{2}}\right).$$
Here $d(x,\partial \Omega):=\inf\{|x-y|, y\in \partial \Omega\}$ and $L_{1}, L_{2},k_{1},k_{2}$ are some constants.

\end{lemma}

    We need to emphasize that the $x$ in the Lemma \ref{U:lemma3.8}  is different from the one in our article.
  Here $x\in \Omega \subset\mathbb{R}^{n}$, and $\Omega $ is a bounded domain of  $ \mathbb{R}^{n}$  with  $C^2$
boundary.
In addition, we introduce $$G(x,\nabla u,\nabla^2u):= \tilde{F}(x,\nabla u,\nabla^2u)+b(x)\cdot \nabla u |\nabla u|^{\alpha},\ \  \alpha>-1,$$ where  $b:\Omega \to \mathbb{R}^{n}$ is continuous and bounded, and  $\tilde{F}:\Omega \times \mathbb{R}^{n}\setminus \{0\}\times S^n \to \mathbb{R}$ is  continuous. Here $S^{n}$ denotes the space of symmetric matrices on $\mathbb{R}^{n}$.

 $\tilde{F}$ and $b(x)$ satisfy the following assumptions:

 $(\text{H}1)\ \  \tilde{F}(x,tP, \mu X) =|t|^{\alpha} \mu \tilde{F}(x,P,X)\ \ \forall t\in \mathbb{R}\backslash \{0\}, \mu\in \mathbb{R}^{+}, x\in \Omega.$

  $(\text{H}2)\ \ \text{There \ exist} \ c_1 \geq c_2>0$ such that
  $$ c_2|P|^{\alpha}tr(N)\leq  \tilde{F} (x,P,M+N)- \tilde{F}(x,P,M)\leq c_1|P|^{\alpha}tr(N),$$
  for all $x\in \Omega, \ P\neq0,$  and  $ (M,N)\in S^n\times S^n$ with $N\geq 0$.

   $(\text{H}3)\ \  $ There exists a continuous function $\tilde{\omega}$ with $\tilde{\omega}(0)=0$ such that for all $x,y, P\neq 0$, and   $X\in S^{n}$,
   $$ \left|\tilde{F}(x,P,X)-\tilde{F}(y,P,X)\right|\leq \tilde{\omega}(|x-y|)|P|^{\alpha}|X|,$$
where $|X|$  is the Frobenius norm of $X$.

    $(\text{H}4)\ \  $ There exists a continuous function $\omega$ with $\omega(0)=0,$ such that if $X\in S^{n}$, $Y\in  S^{n}$ and $\zeta\in \mathbb{R}$  satisfy

    $$-\zeta \begin{pmatrix}
       I&
       0\\
       0&I
    \end{pmatrix}\leq \begin{pmatrix}
       X&0\\
       0&Y
    \end{pmatrix}\leq
     4\zeta\begin{pmatrix}
       I&-I\\
       -I&I
    \end{pmatrix},
    $$
    where $I$ is the identity matrix in $\mathbb{R}^n$, then for all $x,y\in \mathbb{R}^n$, $x\neq y$,
    $$\tilde{F}(x,\zeta(x-y),X)-\tilde{F}(y,\zeta(x-y),-Y) \leq \omega (\zeta|x-y|^2).$$

     $(\text{H}5)\ \  $ The function $b:\Omega \to \mathbb{R}^{n}$ is  continuous and bounded, and satisfies:

      either $\alpha<0$ and $b(x)$ is H\"older continuous  with exponent $1+\alpha$; or $\alpha\geq 0$ and for all $x$, $y\in \Omega$ ,   $ \langle b(x)-b(y), x-y\rangle\leq 0.$

   Next, we present the  existing results in  \cite{HP2}.
   
Suppose  $\overline{{F}} $ is a real-valued continuous function on $ \overline{\Omega}\times \mathbb{R}\times \mathbb{R}^n \times S^n$,
 where $\Omega$ is an  bounded open subset of $\mathbb{R}^n$ and $S^n$ denotes the space of $n\times n$ real symmetric matrices,

 $(\text{$\widetilde{H}$}1)\ \  \overline{{F}}$  is  degenerate elliptic:  for all $(x,r,p,X)\in {\Omega}\times \mathbb{R}\times \mathbb{R}^n \times S^n$, $Y\in S^n$,
 \begin{equation}\label{eq:21}
\overline{{F}}(x,r,p,X+Y)\le \overline{{F}}(x,r,p,X) \qquad \text{if } Y\ge 0,
\end{equation}

 $(\text{$\widetilde{H}$}2)\ \ \text{For all} \ R<\infty$,  there exists a  constant ${\gamma\geq 0}$  such that
  \begin{equation}\label{eq:22}{\overline{{F}}(x, r, p, M)\geq  \overline{{F}} (x, s, p, M) +\gamma(r-s)\ \ \text{for} \ r\geq s}.\end{equation}

  for all $x\in\Omega$, $R\ge r\ge s\ge -R$, $p\in\mathbb{R}^n$, $A\in S^n$.

   $(\text{$\widetilde{H}$}3)\ \ \text{For all} \  R<\infty$, if $x,y\in\Omega$, $|r|\le R$, $p\in\mathbb{R}^n$, $A\in S^n$, \begin{equation}\label{eq:23}
|\overline{{F}}(x,r,p,A)-\overline{{F}}(y,r,p,A)|\le \omega_R\big(|x-y|(1+|p|)\big)
\end{equation}

 Consider
\begin{equation}\label{U:eq:1}
\overline{{F}}(x, u, Du, D^2u) = 0
\end{equation}
For an USC bounded viscosity subsolution $u$ of \eqref{U:eq:1} and a LSC bounded viscosity supersolution $v$ of  \eqref{U:eq:1},  the comparison assertion is:
\begin{equation} \label{eq:46}{ u-v\leq  \sup_{x\in\partial\Omega} \left\{ u^{*}(x) - v_{*}(x) \right\}^{+} \ \ \text{in} \ \ \Omega,}  \end{equation}
where $\displaystyle u^{*}(x):= \limsup_{y\to x, \ y \in \Omega} u(y),\ \  \text{and}\ \   u_{*}(x):=\liminf_{y\to x, \ y \in \Omega}u(y).$
\begin{lemma}[\cite{HP2}, PROPOSITION II.1]\label{lemma 1}
Assume that the comparison assertion \eqref{eq:46} holds for any pair $(u,v)$  with  the properties indicated above and  assume that there exist $u_1, u_2 \in C(\overline{\Omega})$ respectively viscosity subsolution and supersolution of \eqref{U:eq:1}
 such that $u_1 \equiv u_2 \equiv \varphi$ on $\partial\Omega$. Then, there exists a unique viscosity solution $u \in C(\overline{\Omega})$ of \eqref{U:eq:1} such that $u \equiv \varphi$ on $\partial\Omega$ (and $u_1 \leq u \leq u_2$ on $\overline{\Omega}$).\end{lemma}
\begin{lemma}[\cite{HP2}, THEOREM II.1]\label{eq:24}Assume \eqref{eq:21} and \eqref{eq:23}. If \eqref{eq:22} holds for $\gamma>0$, then  for any  USC bounded viscosity subsolution $u$ of \eqref{U:eq:1} and  LSC bounded viscosity supersolution $v$ of  \eqref{U:eq:1}, the  comparison assertion \eqref{eq:46} holds. Furthermore, if $u$   is a strict viscosity  subsolution(resp.
$v$
 a strict viscosity supersolution), the \eqref{eq:46} still  holds for $\gamma\geq 0$.
\end{lemma}
Here, by strict we mean that, for instance in the case of subsolutions, there exists $g \in C(\overline{\Omega})$ such that $u$ is viscosity subsolution of
\[
\overline{{F}}(x, u, Du, D^2u) = g \quad \text{in } \Omega, \quad \text{and } g < 0 \quad \text{in } \Omega.
\]

   \begin{Remark}\label{U:remark3.2}
  We claim that there exists a continuous viscosity solution for
  \begin{equation}\label{U:3.19}
    \begin{cases}
F((t,x),\nabla_\mathbb{B}u,\nabla_\mathbb{B}^2u)+h(u)=f(t,x)  \ \ \ & (t,x)\in  \Omega,\\
u=0  \ & (t,x)\in    \partial{\Omega},
\end{cases}
\end{equation}
when $\Omega$ is a bounded domain of  $\mathbb{B}$  with  $C^2$
boundary with respect to the  cone distance defined in Definition \ref{U:cone 9}. Here $F((t,x),\nabla_\mathbb{B}u,\nabla_\mathbb{B}^2u)$ defined in \eqref{U:eq:76}.

  In fact, let $$a=\ln t,\   \overset{T}{}\Omega =\{(a,x): (t,x)\in \Omega \},$$ and then $\overset{T}{}\Omega$ is a bounded domain of  $ \mathbb{R}^{n}$  with a $C^2$ boundary. Set
  $$  v(a,x)=u(t,x),  \ \overset{T}{}f(a,x)=f(t,x),$$
  $$\overset{T}{}F((a,x),\nabla v(a,x),\nabla^2 v(a,x))=F((t,x),\nabla_{\mathbb{B}} u(t,x),\nabla_{\mathbb{B}}^2u(t,x))=e^{-2a}tr(\nabla^2 v)+e^{-2a}(n-2) \partial_{a}v,$$ and then $tr(\nabla^2 v)+(n-2) \partial_{a}v+e^{2a}h(v(a,x))=\overset{T}{}f(a,x)e^{2a}$ satisfies the assumptions of Lemma \ref{U:lemma3.8} with  $\alpha=0$.

    By Lemma \ref{U:lemma3.8}, there is  a function $v_{1}\in C(\overline{\overset{T}{}\Omega})$ is a viscosity subsolution of the following equation,  and  that $v_{2}\in C(\overline{\overset{T}{}\Omega})$ is a viscosity supersolution of the following equation,
   \begin{equation}\label{U:eq:922}
    \begin{cases}
tr(\nabla^2 v)+(n-2) \partial_{a}v+e^{2a}h(v(a,x))=\overset{T}{}f(a,x)e^{2a}  \ \ \ & (a,x)\in  \overset{T}{}\Omega,\\
v=0  \ & (a,x)\in    \partial{(\overset{T}{}\Omega)}.
\end{cases}
\end{equation} 
For any  viscosity solution $v$ (resp. supersolution, subsolution) of \eqref{U:eq:922},    $v$ is  viscosity solution (resp. supersolution, subsolution) of \eqref{eq:25}
\begin{equation}\label{eq:25}
    \begin{cases}
-tr(\nabla^2 v)-(n-2) \partial_{a}v-e^{2a}h(v(a,x))+\overset{T}{}f(a,x)e^{2a} =0 \ \ \ & (a,x)\in  \overset{T}{}\Omega,\\
v=0  \ & (a,x)\in    \partial{(\overset{T}{}\Omega)}.
\end{cases}
\end{equation} 
Now we consider equation  \begin{equation}\label{eq:27}-tr(\nabla^2 v)-(n-2) \partial_{a}v-e^{2a}h(v(a,x))+\overset{T}{}f(a,x)e^{2a} =0 \ \ \  (a,x)\in  \overset{T}{}\Omega\end{equation}
 Suppose that $v_3$ is an USC bounded viscosity subsolution of \eqref{eq:27} and $v_4$ is a LSC bounded viscosity supersolution  of  \eqref{eq:27}.
For any  $\lambda>0$, $\phi\in C^{2}(\overset{T}{}\Omega)$, if $v_{3}+\lambda a-\phi$ attains  a local maximum at  $(a_0,x_0)$, then
$$\left( -tr(\nabla^2 \phi)-(n-2) \partial_{a}\phi+(n-2)\lambda-e^{2a}h(v_{3}(a,x))+\overset{T}{}f(a,x)e^{2a}\right)\big|_{(a_0,x_0)}\leq 0.$$
Since $a=\ln t \leq 0$,  the  non-increasing property of function  $h$ yields  that
$$\left(- tr(\nabla^2 \phi)-(n-2) \partial_{a}\phi-e^{2a}h(v_{3}(a,x)+\lambda a)+\overset{T}{}f(a,x)e^{2a}\right)\big|_{(a_0,x_0)}\leq -(n-2)\lambda<0.$$
By the Definition \ref{U:D1}, we have that   $v_{3}+\lambda a$ is  the viscosity subsolution of $$-tr(\nabla^2 v)-(n-2) \partial_{a}v-e^{2a}h(v(a,x))+\overset{T}{}f(a,x)e^{2a}=-(n-2)\lambda.$$
   In other words,    $v_{3}+\lambda a$ is a strict viscosity subsolution of $$-tr(\nabla^2 v)-(n-2) \partial_{a}v-e^{2a}h(v(a,x))+\overset{T}{}f(a,x)e^{2a}=0.$$
  Lemma \ref{eq:24} gives
  $$v_{3}+\lambda a-v_{4}\leq \sup \limits_{\partial(\overset{T}{}\Omega)}(v^{*}_{3}+\lambda a-{v_{4}}_{*})^{+}\  \text{in}\  \ \overset{T}{}\Omega. $$  That $\lambda\to 0$ implies $$v_{3}-v_{4}\leq \sup \limits_{\partial(\overset{T}{}\Omega)}(v^{*}_{3}-{v_{4}}_{*})^{+}\  \text{in}\  \ \overset{T}{}\Omega. $$
By Lemma \ref{lemma 1}, there exists a unique viscosity solution $v \in C(\overline{\Omega})$ of \eqref{eq:25}, which is also a viscosity solution  of \eqref{U:eq:922}.

Let $u(t,x)=v(a,x)$ with $a=\ln t$. We will prove that  $u\in C(\overline{\Omega})$  is a viscosity solution for \eqref{U:3.19} in $\Omega$.
 For any $\phi(t,x)\in C^2(\Omega)$, if $u-\phi$ has a local minimum at $(t_0,x_0)\in \Omega$, then $v-\overset{T}{}\phi$ has a local minimum at $(a_0,x_0)\in \overset{T}{}\Omega$ with $a_0=\ln t_0 $, $\overset{T}{}\phi(a,x)=\phi(t,x)$. Since $v$ is a viscosity solution of \eqref{U:eq:922}, we have
 \begin{equation}\label{U:3.21}
\left(e^{-2a}tr(\nabla^2(\overset{T}{}\phi))+(n-2)  e^{-2a} \partial_{a}(\overset{T}{}\phi)+h(v)\right)\big|_{(a_0,x_0)}=\overset{T}{}f(a_0,x_0).\end{equation}
 Since $\overset{T}{}\phi(a,x)=\phi(t,x)$ with $a=\ln t$, similar to \eqref{U:eq:866}, we get
\begin{equation}\label{U:3.22}
\nabla_{\mathbb{B}}\phi(t,x)=\nabla(\overset{T}{}\phi)(a,x), \ \nabla_{\mathbb{B}}^2\phi(t,x)=\nabla^2(\overset{T}{}\phi)(a,x).
\end{equation}
 Substituting $t=e^a$ and $\eqref{U:3.22}$ into \eqref{U:3.21}, we have
 $$F((t_0,x_0),\nabla_\mathbb{B}\phi(t_0,x_0),\nabla_\mathbb{B}^2\phi(t_0,x_0))+h(u(t_0,x_0))\leq f(t_0,x_0).$$
 Since $v=0$ on $\partial (\overset{T}{}\Omega)$,  $u(t,x)=0$ on $\partial \Omega$.
  Hence $u(t,x)$ is a viscosity supersolution for \eqref{U:3.19}. Similarly, we get $u(t,x)$ is also a viscosity subsolution for \eqref{U:3.19}. Hence the claim is true.
  \end{Remark}

 \begin{proof}[\textbf{Proof of Theorem \ref{U:B}}]

Let $\{H_{j}\}$  be  defined as in  Assumption \ref{U:assumption2.2}.
By  Remark \ref{U:remark3.2}, we can find a viscosity solution $u_j\in C(\overline{H}_j)$ for
\begin{equation}\label{U:H9}
    \begin{cases}
F((t,x),\nabla_\mathbb{B}u,\nabla_\mathbb{B}^2u)+h(u)=f(t,x) \ \ \ & (t,x)\in  H_{j},\\
u=0  \ & (t,x)\in    \partial{H_{j}},
\end{cases}
\end{equation}
where $F((t,x),\nabla_\mathbb{B}u,\nabla_\mathbb{B}^2u)$ defined as in \eqref{U:eq:76}.
By Theorem \ref{U:C}, we can get
\begin{equation}\label{U:3.18}\begin{aligned}
||u_j(t,x)||_{\rho,\overline{H}_j}&\leq  C \left( {\vert\vert t^2f\vert\vert_{\mathbb{L}^{1}_{n}\left(H_j\right)}} +h(-C^*K_0d_0 {\vert\vert t^2f\vert\vert_{\mathbb{L}^{1}_{n}\left(H_j\right)}})\right) \\ &\leq C \left( {\vert\vert t^2f\vert\vert_{\mathbb{L}^{1}_{n}\left(\mathbb{B}\right)}} +h(-C^*K_0d_0 {\vert\vert t^2f\vert\vert_{\mathbb{L}^{1}_{n}\left(\mathbb{B}\right)}})\right),
 \end{aligned}\end{equation}
where $C$ doesn't depend on $j$.
 Therefore, by Arzela-Ascoli theorem we can find a subsequence of $\{u_{j_k}\}$ that converges uniformly to $v_{\bar{O}}$ in $\bar{O}$, where $O\Subset \mathbb{B}$ is a bounded open set.

 Next, we will prove this theorem in three steps.

Step 1:
 we claim that $v_{\bar{O}}$ is the viscosity solution of \eqref{U:equation3.12} in $ O$,
 \begin{equation}\label{U:equation3.12}
 F((t,x),\nabla_\mathbb{B}u,\nabla_\mathbb{B}^2u)+h(u)=f(t,x) .\end{equation}
 We will only verify that $v_{\bar{O}}$ is the viscosity  supersolution of \eqref{U:equation3.12}, because the  proof for the case that $v_{\bar{O}}$ is the viscosity  subsolution is  similar.

 In fact, for all $\phi\in C^{2}(O)$, if $v_{\bar{O}}-\phi$ has a local minimum at $(t_0,x_0)\in O$, then $$v_{\bar{O}}-\phi+|(t,x)-(t_0,x_0)|_{\mathbb{B}}^4$$ reaches a strict local minimum at $(t_0,x_0)$. For convenience we define $\tilde{\phi}=\phi-|(t,x)-(t_0,x_0)|_{\mathbb{B}}^4$, and then
$v_{\bar{O}}-\tilde\phi$ has a strict local minimum at $(t_0,x_0)\in O$. Hence  we have
 $$\mathop{\min}\limits_{\partial U}(v_{\bar{O}}-\tilde\phi)(t,x)>(v_{\bar{O}}-\tilde\phi)(t_0,x_0)$$ for some open neighborhood $U$
of $(t_0, x_0)$.

Since $u_{j_k}$ converges uniformly to $v_{\bar{O}}$ in $\bar{O}$,  when $j_k$ is sufficiently large, we have
$$\mathop{\min}\limits_{\partial U}(u_{j_k}-\tilde\phi)(t,x)>(u_{j_k}-\tilde\phi)(t_0,x_0) . $$
Hence $u_{j_k}-\tilde\phi$ has a local minimum at $ (t_{j_k},x_{j_k})\in U$. Now, if we let the radius of $U$ go to zero, we get a sequence of points $(t_{j_k},x_{j_k})\to (t_0,x_0)$ such that $u_{j_k}-\tilde\phi$ has a local minimum at $ (t_{j_k},x_{j_k})$.  When $j_k$ is large enough, $u_{j_k}$ is the viscosity supersolution  of  \eqref{U:equation3.12} in $O$, and then we have
$$  F((t_{j_k},x_{j_k}),\nabla_\mathbb{B}\tilde\phi(t_{j_k},x_{j_k}),\nabla_\mathbb{B}^2\tilde\phi(t_{j_k},x_{j_k}))+h(u_{j_k}(t_{j_k},x_{j_k}))\leq f(t_{j_k},x_{j_k}) $$
with $$\nabla_\mathbb{B}\tilde\phi(t_{j_k},x_{j_k})=\nabla_\mathbb{B}\phi(t_{j_k},x_{j_k})-\nabla_\mathbb{B}|(t_{j_k},x_{j_k})-(t_0,x_0)|_{\mathbb{B}}^4,$$

$$ \nabla^2_\mathbb{B}\tilde\phi(t_{j_k},x_{j_k})=\nabla_\mathbb{B}\phi(t_{j_k},x_{j_k})-\nabla^2_\mathbb{B}|(t_{j_k},x_{j_k})-(t_0,x_0)|_{\mathbb{B}}^4.$$

   Let $(t_{j_k},x_{j_k})\to (t_0,x_0)$. Since  $u_{j_{k}}$ converges uniformly to $v_{\bar{O}}$ on $\bar{O}$,  together with the continuity of $F,f$ and $h$,  and the $C^2$ smoothness of $\tilde{\phi}$, we obtain that
$$F((t_0,x_0),\nabla_\mathbb{B}\phi(t_0,x_0),\nabla_\mathbb{B}^2\phi(t_0,x_0))+h(v_{\bar{O}}(t_0,x_0))\leq f(t_0,x_0).$$
 Hence  $v_{\bar{O}}$ is the viscosity supersolution of \eqref{U:equation3.12} in $ O$.

Step 2 :  We  construct a function $v$ such that it is the viscosity solution of \eqref{U:equation3.12}
  in $ \mathbb{B}$.

   Let $\bar{O}_{i}$ be a sequence of compact sets such that
$$O_{i}\Subset O_{i+1}\Subset \mathbb{B} \ \ \ \text{and } \ \ \ \cup_{i} O_{i} =\mathbb{B},$$ and we define
\begin{equation}
v=\lim\limits_{i\to\infty}v_{\bar{O}_i}.\end{equation}


 For any $(t_0,x_0)\in \mathbb{B}$, we can choose $i$  large enough such that $(t_0,x_0)\in O_i$.
If $v-{\phi}$ attains a  local minimum at $(t_0,x_0)\in \mathbb{B}$ for any ${\phi}\in C^2(\mathbb{B})$, then $v_{\bar{O}_i}-{\phi}$ attains  a  local minimum at $(t_0,x_0)\in \mathbb{B}$. So the following inequality holds $$F((t_0,x_0),\nabla_\mathbb{B}\phi(t_0,x_0),\nabla_\mathbb{B}^2\phi(t_0,x_0))+h(v(t_0,x_0))\leq f(t_0,x_0),
 $$
which implies $v$ is the viscosity supersolution of \eqref{U:equation3.12}
  in $ \mathbb{B}$. Similarly, $v$ is also the viscosity subsolution \eqref{U:equation3.12}
  in $ \mathbb{B}$.

  Step 3:  We  define the value of  $v$ on $\partial \mathbb{B}$.   And we  prove that  $v$ is in  $C(\overline{\mathbb{B}})$   which satisfies \eqref{U:H8} and  \eqref{U:T1.10}.

  For any $\varepsilon>0$,  for all $(t_0,x_0)\in \mathbb{B}$, there is a $u_{j_{k}}$, a viscosity solution of \eqref{U:H9} in $H_{j_k}$, such that
  \begin{equation}\label{U:eq:87}
  |v(t_0,x_0)-u_{j_k}(t_0,x_0)|\leq \frac{\varepsilon}{3}. \end{equation}

   Indeed,  for all $(t_0,x_0)\in \mathbb{B}$, there exists $ O_i$ such that $(t_0,x_0)\in O_i$. Choosing $j_k$  large enough  such that $\bar{O}_i\subset H_{j_k}.$ Since $\{u_{j_k}\}$ converges uniformly to $v_{\bar{O}_i}$ in $\bar{O_i}$,   \eqref{U:eq:87} holds when $j_k$ is large enough.

 Recall  $d_{\partial\mathbb{B}}(t,x)$ defined in \eqref{U:eq:85}.  When $$d_{\partial\mathbb{B}}(t_0,x_0)\leq (\frac{\varepsilon}{3C})^{\frac{1}{\rho}}{ \left( {\vert\vert t^2f\vert\vert_{\mathbb{L}^{1}_{n}\left(\mathbb{B}\right)}} +h(-C^*K_0d_0 {\vert\vert t^2f\vert\vert_{\mathbb{L}^{1}_{n}\left(\mathbb{B}\right)}})\right)}^{\frac{-1}{\rho}},$$ by \eqref{U:3.18}, \eqref{U:eq:87} and  $u_{j_k}=0$ on $\partial H_{j_k}$, we have
 $$
 \begin{aligned}|v(t_0,x_0)-0|&\leq  |v(t_0,x_0)-u_{j_k}(t_0,x_0)|+ |u_{j_k}(t_0,x_0)-0| \\ &\leq
 \frac{\varepsilon}{3}+ C \left( {\vert\vert t^2f\vert\vert_{\mathbb{L}^{1}_{n}\left(\mathbb{B}\right)}} +h(-C^*K_0d_0 {\vert\vert t^2f\vert\vert_{\mathbb{L}^{1}_{n}\left(\mathbb{B}\right)}})\right)d_{\mathbb{B}}^{\rho}((t_0,x_0),\partial H_{j_k})\\
 &\leq  \frac{\varepsilon}{3}+ C \left( {\vert\vert t^2f\vert\vert_{\mathbb{L}^{1}_{n}\left(\mathbb{B}\right)}} +h(-C^*K_0d_0 {\vert\vert t^2f\vert\vert_{\mathbb{L}^{1}_{n}\left(\mathbb{B}\right)}})\right)d_{\partial\mathbb{B}}^{\rho}(t_0,x_0)
 \leq\frac{\varepsilon}{3}+\frac{\varepsilon}{3}<\varepsilon,
  \end{aligned}$$
  where $d_{\mathbb{B}}((t_0,x_0),\partial H_{j_k})$ is the cone distance from $(t_0,x_0)$  to $\partial H_{j_k}$.
 That is,
 $ \lim\limits_{(t_0,x_0)\to \partial\mathbb{B}}v(t_0,x_0)=0.$
 We can define $v\big|_{\partial\mathbb{B}}=0$,
 and then $v$ is in $ C(\mathbb{\overline{B}})$ and satisfies \eqref{U:H8}.

 From \eqref{U:3.18}, we have $|u_j|\leq C \left( {\vert\vert t^2f\vert\vert_{\mathbb{L}^{1}_{n}\left(\mathbb{B}\right)}} +h(-C^*K_0d_0 {\vert\vert t^2f\vert\vert_{\mathbb{L}^{1}_{n}\left(\mathbb{B}\right)}})\right).$
 From the construction of $v$, we get  $|v|\leq C \left( {\vert\vert t^2f\vert\vert_{\mathbb{L}^{1}_{n}\left(\mathbb{B}\right)}} +h(-C^*K_0d_0 {\vert\vert t^2f\vert\vert_{\mathbb{L}^{1}_{n}\left(\mathbb{B}\right)}})\right),$ and then $v$ satisfies
 $\mathop{\sup }\limits_{\mathbb{{B}}}\vert v \vert < +\infty$.
Finally we obtain \eqref{U:T1.10} by Theorem \ref{U:C}.
\end{proof}

\section{\textbf{The existence of weak solutions}}
In this section, when $f \in C(\overline{\mathbb{B}})\cap L^{\infty}(\mathbb{B})$,
 we conclude that the bounded  viscosity solution of \eqref{U:eq:11} is  also  a  weak solution of \eqref{U:eq:11}, which guarantees the existence of weak solutions to \eqref{U:H8}. Finally, we  prove that  the continuous   weak  solution of \eqref{U:eq:11} is  also a  viscosity solution of \eqref{U:eq:11}.

\begin{definition}[Inf-convolution]\label{U:def-2}
We define the inf-convolution of a function u as
 \begin{equation}\label{U:eq:95}
 u_{\epsilon}(t,x):=\mathop{\inf }\limits_{(s,y)\in\mathbb{B}}\left(u(s,y)+\frac{|(t,x)-(s,y)|_{\mathbb{B}}^2}{2\epsilon}\right).\end{equation}

\end{definition}

\begin{lemma}[\cite{VP}, Lemma A.1, A.2]\label{U:Lemma2.1} Let $\Omega$ be a domain of $\mathbb{R}^{n}$. Assume that $u:{\Omega}\to \mathbb{R}$ is bounded and lower semicontinuous and $u_{\epsilon, E}$  defined as
$$ u_{\epsilon,E}(x):=\mathop{\inf }\limits_{y\in \Omega}\left(u(y)+\frac{|x-y|^2}{2\epsilon}\right).$$
The following assertions hold.

$(1)$ The sequence $\{u_{\epsilon,E}\}$ is increasing  and converges pointwise to $u$ in $\Omega$;

$(2)$  $u_{\epsilon,E}$ can be written as
 \begin{equation}\label{U:eq:59}
 u_{\epsilon,E}(x):=\mathop{\inf }\limits_{y\in { \Omega \cap B(x,{r(\epsilon)})}}\left(u(y)+\frac{|x-y|^2}{2\epsilon}\right),\end{equation}
 where $r(\epsilon)=2\sqrt{\epsilon||u||_{L^{\infty}(\Omega)}}$, and $B(x,{r(\epsilon)})$ is a Euclid ball  with $x$ as  center and $r(\epsilon)$ as  radius;

$(3)$ $u_{\epsilon,E}(x)-\frac{|x|^2}{2\epsilon}$ is concave in $\Omega_{r(\epsilon),E}$,  where
\begin{equation}\label{U:eq:20}\Omega_{r(\epsilon),E}=\{x\in \Omega \ \big|\ |x-y|>r(\epsilon),\  \text{for\ all}\  y\in \partial \Omega\};\end{equation}

$(4)$ 
 $u_{\epsilon} $ is locally Lipschitz continuous  in the interior of $\Omega_{r(\epsilon),E}$.
\end{lemma}
For these and further properties, see \cite{VP}.

\begin{lemma}\label{U:lemma 6.2}
Assume that $u:{\mathbb{B}}\to \mathbb{R}$ is bounded and lower semicontinuous and $u_{\epsilon}$  defined as in Definition \ref{U:def-2}.
The following assertions hold.

$(1)$ The sequence $\{u_{\epsilon}\}$ is increasing  and converges pointwise to $u$ in $\mathbb{B}$;

$(2)$  $u_{\epsilon}$ can be written as
 \begin{equation}\label{U:eq:108}
 u_{\epsilon}(t,x):=\mathop{\inf }\limits_{(s,y)\in { \mathbb{B} \cap B_{\mathbb{B}}((t,x),{r(\epsilon)})}}\left(u(s,y)+\frac{|(t,x)-(s,y)|_{\mathbb{B}}^2}{2\epsilon}\right),\end{equation}
 where $r(\epsilon)=2\sqrt{\epsilon||u||_{L^{\infty}(\mathbb{B})}}$, and $B_{\mathbb{B}}((t,x),{r(\epsilon)})$ is a cone ball  with $(t,x)$ as  center and $r(\epsilon)$ as  radius;

$(3)$ $u_{\epsilon}(t,x)-\frac{|(\ln t,x)|^2}{2\epsilon}$ is concave in $\mathbb{B}_{r(\epsilon)}$ about $(\ln t,x)$,  where \begin{equation}\label{U:eq:82}\mathbb B_{r(\epsilon)}=\{(t,x)\in  \mathbb B\ \big{|}\ |(t,x)-(s,y)|_{\mathbb{B}}>r(\epsilon),\ \text{for all}\ (s,y)\in \partial  \mathbb B \};\end{equation}

$(4)$ 
 $u_{\epsilon} $ is locally Lipschitz continuous  in the interior of $\mathbb{B}_{r(\epsilon)}$  about $(\ln t,x)$.
\end{lemma}

\begin{proof}We take (3) as an example to conduct the proof, and the remaining proofs are analogous.
Let $$a=\ln t,\   \overset{T}{}\mathbb{B} =\{(a,x): (t,x)\in \mathbb{B} \}.$$  Then  $\overset{T}{}\mathbb{B}$ is an open domain of  $ \mathbb{R}^{n}$. Let
$\partial_{E}  (\overset{T}{}\mathbb{B})$
 denote the boundary of
$\overset{T}{}\mathbb{B}$
 with respect to the Euclidean metric. By \eqref{U:eq:20} and \eqref{U:eq:82}, we have \begin{equation}\label{U:eq:25}\begin{aligned}(\overset{T}{}\mathbb{B})_{r(\epsilon),E}=&\{(a,x)\in  \overset{T}{}\mathbb{B}\ \big{|}\ |(a,x)-(b,y)|>r(\epsilon),\  \text{for all}\ \ (b,y)\in \partial_{E}  (\overset{T}{}\mathbb{B})\}\\
=&\{(a,x): (t,x)\in \mathbb{B}\   \text{and}\  |(t,x)-(s,y)|_{\mathbb{B}}>r(\epsilon),\  \text{for all}\  (s,y)\in \partial \mathbb{B}\}\\=&\{(a,x): (t,x)\in \mathbb{B}_{r(\epsilon)} \} \end{aligned}\end{equation} with $b=\ln s$.  Set
    \begin{equation}\label{U:eq:21}\overset{T}{}u(a,x)=u(t,x), \ \overset{T}{}(u_{\epsilon})(a,x)=u_{\epsilon}(t,x).\end{equation}
By \eqref{U:eq:59}, \eqref{U:eq:108}, \eqref{U:eq:21} and Definition \ref{U:cone 9}, we  easily obtain
\begin{equation}\label{U:eq:23}\begin{aligned}
 \overset{T}{}(u_{\epsilon})(a,x)&=u_{\epsilon}(t,x)=\mathop{\inf }\limits_{(s,y)\in { \mathbb{B} \cap B_{\mathbb{B}}((t,x),{r(\epsilon)})}}\left(u(s,y)+\frac{|(t,x)-(s,y)|_{\mathbb{B}}^2}{2\epsilon}\right)\\ &=\mathop{\inf }\limits_{(b,y)\in { \overset{T}{}\mathbb{B} \cap B((a,x),{r(\epsilon)})}}\left(\overset{T}{}u(b,y)+\frac{|(a,x)-(b,y)|^2}{2\epsilon}\right)=({\overset{T}{}u})_{\epsilon,E}(a,x)
 \end{aligned}\end{equation} with $b=\ln s$.
  Applying Lemma \ref{U:Lemma2.1} to $({\overset{T}{}u})_{\epsilon,E}(a,x)$ on  $ \overset{T}{}\mathbb{B} $ yields  that $\displaystyle({\overset{T}{}u})_{\epsilon,E}(a,x)-\frac{|(a,x)|^2}{2\epsilon}$ is concave in $(\overset{T}{}\mathbb{B})_{r(\epsilon),E}$. That is,
 $\displaystyle u_{\epsilon}(t,x)-\frac{|(\ln t,x)|^2}{2\epsilon}$ is concave in $\mathbb{B}_{r(\epsilon)}$ about $(\ln t,x)$.
\end{proof}

%
%
%
\begin{lemma}\label{U:Lemma5.2}
Suppose that $u:\mathbb B\to\mathbb{R}$ is bounded in $\mathbb B$.
If $u$ is a viscosity supersolution to \eqref{U:eq:11} in $\mathbb B$, then $u_{\epsilon}$ defined  in Definition \ref{U:def-2}  satisfies
 \begin{equation}\label{U:eq:99}
 div_ \mathbb{B}(\nabla_ \mathbb{B}u_{\epsilon})
     +(n-2)(t\partial_t u_{\epsilon})\leq \left(t^2f(t,x) -t^2h(u_{\epsilon})\right)_{\epsilon}
      \end{equation}
     a.e. in $\mathbb B_{r(\epsilon)}$, where $\mathbb B_{r(\epsilon)}$ is  defined  in  \eqref{U:eq:82}  and $$ \left(t^2f(t,x) -t^2h(\cdot)\right)_{\epsilon}:=\sup\limits_{(s,y)\in B_{\mathbb{B}}((t,x),r(\epsilon))}\left(s^2f(s,y) -s^2h(\cdot)\right).$$

\end{lemma}
\begin{proof}
Because $u$ is bounded and lower semicontinuous, we apply  Lemma \ref{U:lemma 6.2} to $u$ on $\mathbb{B}$ and obtain
 \begin{equation}\label{U:eq:88}
 u_{\epsilon}(t,x):=\mathop{\inf }\limits_{(s,y)\in { \mathbb{B} \cap B_{\mathbb{B} }((t,x),{r(\epsilon)})}}\left(u(s,y)+\frac{|(t,x)-(s,y)|_{\mathbb{B}}^2}{2\epsilon}\right)\end{equation}
for $r(\epsilon)=2\sqrt{\epsilon||u||_{L^{\infty}(\mathbb{B})}}$.
For any $\phi\in C^2(\mathbb B_{r(\epsilon)})$, let $u_{\epsilon}-\phi$ has a local minimum at $(t_0,x_0)\in \mathbb B_{r(\epsilon)}$. And then there exists a $r>0$ such that  $B_{\mathbb{B} }((t_0,x_0),r)\Subset \mathbb{B}_{r(\epsilon)} $  and  that
\begin{equation}\label{U:eq:32}(u_{\epsilon}-\varphi)(t,x)\geq (u_{\epsilon}-\varphi)(t_0,x_0),\ \ \ \forall (t,x)\in B_{\mathbb{B} }((t_0,x_0),r).\end{equation}

 By \eqref{U:eq:88}, there is a $(s_0,y_0)\in \overline{ B_{\mathbb{B} }((t_0,x_0), {r(\epsilon)})}$
such that
\begin{equation}\label{U:eq:109} u_{\epsilon}(t_0,x_0)\geq u(s_0,y_0)+\frac{|(t_0,x_0)-(s_0,y_0)|_{\mathbb{B} }^2}{2\epsilon}.\end{equation}
For any $(t,x)\in B_{\mathbb{B} }((t_0,x_0),r)$,  we have $\overline {B_{\mathbb{B} }((t,x), {r(\epsilon)})}\subset \mathbb{B}$. By $\eqref{U:eq:95}$, we have
$$u_{\epsilon}(t,x)\leq u(s,y)+\frac{|(t,x)-(s,y)|_{\mathbb{B} }^2}{2\epsilon} $$
for all $(s,y)\in{ \overline {B_{\mathbb{B} }((t,x), {r(\epsilon)})}} $, and then we have
 $$ u(s,y)+\frac{|(t,x)-(s,y)|_{\mathbb{B} }^2}{2\epsilon} -\phi(t,x)\geq u(s_0,y_0)+\frac{|(t_0,x_0)-(s_0,y_0)|_{\mathbb{B} }^2}{2\epsilon}-\phi(t_0,x_0).$$
Choosing  $\displaystyle s=\frac{ts_0}{t_0}$, $y=x+y_0-x_0$ gives that
$$ u(s,y)-\phi(\frac{st_0}{s_0},y-y_0+x_0)\geq  u(s_0,y_0)-\phi(\frac{s_0t_0}{s_0},y_0-y_0+x_0)$$ for all
 $(s,y)\in B_{\mathbb{B} }((s_0,y_0),r)$.
 It implies that $$(s,y)\big|\ \to u(s,y)-\phi(\frac{st_0}{s_0},y-y_0+x_0)$$ has a local minimum at $(s_0,y_0)$.
 Since $u$ is a viscosity supersolution of \eqref{U:eq:11}, we get
 $$\begin{aligned}
\left(div_ \mathbb{B}(\nabla_ \mathbb{B}\phi)
     +(n-2)(t\partial_t \phi)\right)\big|_{(t_0,x_0)}\leq s_0^2f(s_0,y_0)-s_0^2h(u(s_0,y_0))
     \end{aligned} .$$
 Combining the non-increasing property of $h$, the continuity of $f$ and $h$,  and   \eqref{U:eq:109},  we deduce that
  $$\begin{aligned}
\left(div_ \mathbb{B}(\nabla_ \mathbb{B}\phi)
     +(n-2)(t\partial_t \phi)\right)\big|_{(t_0,x_0)}\leq s_0^2f(s_0,y_0) -s_0^2h(u_{\epsilon}(t_0,x_0))\leq \left(t_0^2f(t_0,x_0) -t_0^2h(u_{\epsilon})\right)_{\epsilon},
     \end{aligned}$$
    that is, $u_{\epsilon}$ is a viscosity supersolution of \eqref{U:eq:99}
     in $\mathbb B_{r(\epsilon)}$.
By  Lemma \ref{U:lemma 6.2},  we know $u_{\epsilon}$  is twice differentiable a.e. in $\mathbb B_{r(\epsilon)}$. Then \eqref{U:eq:99} holds  a.e. in $\mathbb B_{r(\epsilon)}$.

\end{proof}
\begin{lemma}\label{U:Lemma5.3}
 If $u$ is a viscosity supersolution to \eqref{U:eq:11} and  bounded in $\mathbb{B}$, then  the following inequality holds  for any non-negative $\psi\in C_{0}^{\infty}( \mathbb{B})  $  satisfying $\operatorname{supp} \psi \subset \mathbb{B}_{r(\epsilon)}$,
  \begin{equation}\label{U:eq:110}
\int_{\mathbb{B}}\nabla_ \mathbb{B} {u}_{\epsilon}\cdot \nabla_ \mathbb{B}\psi \frac{dt}{t}dx \geq \int_{\mathbb{B}}-\left(div_ \mathbb{B}(\nabla_ \mathbb{B}{u}_{\epsilon})\right)\psi \frac{dt}{t}dx,
\end{equation}
where $u_{\epsilon}$ is   defined in Definition \ref{U:def-2} and  $\mathbb{B}_{r(\epsilon)}$ is  defined in \eqref{U:eq:82}.
\end{lemma}
\begin{proof}
Let $$a=\ln t,\   \overset{T}{}\mathbb{B} =\{(a,x): (t,x)\in \mathbb{B} \}.$$ Then  $\overset{T}{}\mathbb{B}$ is an open domain of  $ \mathbb{R}^{n}$ and  \eqref{U:eq:25} implies that  $$ (\overset{T}{}\mathbb{B})_{r(\epsilon),E}=\{(a,x): (t,x)\in \mathbb{B}_{r(\epsilon)} \} . $$ Set
  $   \overset{T}{}u(a,x)=u(t,x), \ \overset{T}{}(u_{\epsilon})(a,x)=u_{\epsilon}(t,x),\  \overset{T}{}\psi(a,x)=\psi(t,x).$
By calculations analogous to\eqref{U:eq:866}, we have
  \begin{equation}\label{U:eq:24}\nabla(\overset{T}{}\psi) (a, x)=\nabla_{\mathbb{B}}\psi (t, x). \end{equation}
 The equation \eqref{U:eq:23} implies that  \begin{equation}\label{U:eq:116} ({\overset{T}{}u})_{\epsilon,E}(a,x)= \overset{T}{}(u_{\epsilon})(a,x) = u_{\epsilon}(t,x). \end{equation}
Analogous to \eqref{U:eq:24}, we have\begin{equation}\label{U:eq:26} \nabla({\overset{T}{}u})_{\epsilon,E}(a,x) =\nabla_{\mathbb{B}}u_{\epsilon}(t,x), \ \  div\left(\nabla({\overset{T}{}u})_{\epsilon,E}\right) =div_ \mathbb{B}(\nabla_ \mathbb{B}{u}_{\epsilon}) .\end{equation}
For $\overset{T}{}u(a,x)$,  by the proof of  Theorem 3.1 in  \cite{VP}, we find that  for any non-negative $\phi(a,x)\in  C_{0}^{\infty}((\overset{T}{}\mathbb{B})_{r(\epsilon),E})$,
 $$
\int_{\overset{T}{}\mathbb{B} }|\nabla({\overset{T}{}u})_{\epsilon,E}|^{p-2}\nabla({\overset{T}{}u})_{\epsilon,E}\cdot \nabla \phi dadx \geq \int_{\overset{T}{}\mathbb{B} }-\phi div\left(|\nabla({\overset{T}{}u})_{\epsilon,E}|^{p-2}\nabla({\overset{T}{}u})_{\epsilon,E}\right) dadx , \ \   \forall p\geq 2.
$$
Choosing   $p=2$ and  $\phi(a,x)= \overset{T}{}\psi(a,x)$   gives
$$\int_{\overset{T}{}\mathbb{B} }\nabla({\overset{T}{}u})_{\epsilon,E}\cdot \nabla (\overset{T}{}\psi) dadx \geq \int_{\overset{T}{}\mathbb{B} }-\overset{T}{}\psi div\left(\nabla({\overset{T}{}u})_{\epsilon,E}\right) dadx. $$
Via a simple variable substitution, we deduce \eqref{U:eq:110} from equations \eqref{U:eq:24},  \eqref{U:eq:116} and \eqref{U:eq:26}.
\end{proof}

\begin{lemma}[\cite{IF22}, Remark 1]\label{U:Lemma5.4} Suppose $\Omega$ is a bounded domain with  $C^{2}$ boundary, and then the distance to the boundary
$$ d(x,\partial \Omega):=\inf\{|x-y|, y\in \partial \Omega\}$$
satisfies that $d(x,\partial \Omega)$ is Lipschitz continuous and  $|\nabla d(x,\partial \Omega)|=1$.

\end{lemma}
\begin{lemma}[\cite{LCE}, Theorem 4 in Chapter 5]\label{U:Lemma5.5}
 Let $\Omega$ be open and bounded, with $\partial\Omega$ of class $C^1$. Then $u: \Omega\to \mathbb{R}$ is Lipschitz continuous if only if $u\in W^{1,\infty}(\Omega)$.
\end{lemma}
\begin{lemma}\label{U:lemma5.7}
For any $\gamma\in \mathbb{R}$, $ H^{1}$ norm and $\mathbb{H}_{2}^{1,\gamma}$ norm are equivalent in any compact  $K\Subset \mathbb{B}$. That is, for any  $u\in \mathbb{H}_{2}^{1,\gamma}(K)$,  there are constants  $C_1>0$ and $C_2>0$ such that
\begin{equation}\label{U:eq:57}
C_1||u||_{H^{1}(K)}\leq ||u||_{\mathbb{H}_{2}^{1,\gamma}(K)}\leq C_2||u||_{H^{1}(K)}.\end{equation}
\end{lemma}

\begin{proof}
Since $X$ is bounded,  and $t$  is bounded   in $K$,   there are $C_1'>0$ and $C_2'>0$ such that
 \begin{equation}\label{U:eq:54}
 C_1'\leq\left(t(1-t)d(x,\partial X)\right)^{n-2\gamma } t^{2\alpha -1}\leq C_2', \ \ \ \text{if} \ \ \  0\leq \alpha\leq 1.\end{equation}

Recalling the Definition \ref{U:Definition2.2}, we have
$$\begin{aligned}
||u||_{\mathbb{H}_{2}^{1,\gamma}(K)}=&\sum_{\alpha+|\beta|\leq1}
\big(\int_{K}|{\left(t(1-t)d(x,\partial X)\right)}^{\frac{n}{2}-\gamma}(t\partial _t)^{\alpha}\partial_x^{\beta} u(t,x)|^2 \frac{dt}{t}dx\big)^{\frac{1}{2}}\\ =&\sum_{\alpha+|\beta|\leq1}\big(\int_{K}\left(t(1-t)d(x,\partial X)\right)^{n-2\gamma } t^{2\alpha -1}
|\partial _t^{\alpha}\partial_x^{\beta} u(t,x)|^2 dtdx\big)^{\frac{1}{2}}.
\end{aligned}
$$
By \eqref{U:eq:54} we get
$$ (C_1')^{\frac{1}{2}}||u||_{H^{1}(K)}\leq||u||_{\mathbb{H}_{2}^{1,\gamma}(K)}\leq (C_2')^{\frac{1}{2}}||u||_{H^{1}(K)},$$
which implies  that \eqref{U:eq:57} holds with $C_1=(C_1')^{\frac{1}{2}}$, $C_2=(C_2')^{\frac{1}{2}}$.
\end{proof}

\begin{proof}[\textbf{Proof of Proposition \ref{U:Theorem 1.5}}]
By  Lemma \ref{U:Lemma5.2} and  \ref{U:Lemma5.3}, we can obtain
 \begin{equation}\label{U:1.99}
\int_{\mathbb{B}}\nabla_ \mathbb{B} {u}_{\epsilon}\cdot \nabla_ \mathbb{B}\psi \frac{dt}{t}dx  \geq \int_{\mathbb{B}} \left(-\left(t^2f(t,x) -t^2h(u_{\epsilon})\right)_{\epsilon}+(n-2)t\partial_t {u}_{\epsilon}\right)\psi\frac{dt}{t}dx
\end{equation}
for any non-negative $\psi \in  C^{\infty}_{0}(\mathbb{B}) $  satisfying $\operatorname{supp} \psi \subset \mathbb{B}_{r(\epsilon)}$.  $\mathbb{B}_{r(\epsilon)}$  is defined in \eqref{U:eq:82}.

 We  will divide the proof in three steps.

Step 1:  For any non-negative $\psi \in  C^{\infty}_{0}(\mathbb{B}_{r(\epsilon)}),$   we define $\bar{\psi}$ as follows:
\begin{equation}\label{U:eq:117} \bar{\psi}={\left(t(1-t)d(x,\partial X)\right)}^2\psi, \end{equation}
 where $d(x,\partial X)$ is defined in Definition \ref{U:def-1}. And we  claim that
$\bar{\psi}\in \mathbb{H}_{2,0}^{1,\frac{n}{2}}(\mathbb{B}_{r(\epsilon)})$, and that \eqref{U:1.99} still  holds  for $\bar{\psi}$,  defined in \eqref{U:eq:117}.

Since $X$ is a bounded  open set with smooth boundary, by Lemma \ref{U:Lemma5.4}
we have  $|\nabla d(x,\partial X)|=1$ and  $d(x,\partial X)$  is Lipschitz continuous. In addition, by Lemma \ref{U:Lemma5.5} we have
$d(x,\partial X) \in  W^{1,\infty}(X)$. Together with
$t^2(1-t)^2\psi\in C^{\infty}_{0}(\mathbb{B}_{r(\epsilon)})$, we get that the weak derivative of $\bar{\psi}$ exists and $\bar{\psi}\in \mathbb{H}_{2}^{1,\frac{n}{2}}(\mathbb{B}_{r(\epsilon)})$.

Since $\bar{\psi}\in \mathbb{H}_{2}^{1,\frac{n}{2}}(\mathbb{B}_{r(\epsilon)})$ and  $\text{supp}  \bar{\psi}\subset \mathbb{B}_{r(\epsilon)}$,  there is a compact $K$ such that $\text{supp}  \bar{\psi} \Subset K \Subset \mathbb{B}_{r(\epsilon)} $. By Lemma \ref{U:lemma5.7}, we have $\bar{\psi}\in H^{1}_{0}(K)$. Furthermore, since $\bar{\psi}\geq 0$,  there is a sequence of non-negative  $\{\psi_j\}$ such that $\psi_j\in C_{0}^{\infty}(K)$ and
\begin{equation}\label{U:eq:16}||\psi_j-\bar{\psi}||_{H^{1}(K)} \to 0, \  \text{as} \ j\to \infty.\end{equation}

By \eqref{U:eq:57},  we have \begin{equation}\label{U:eq:17}||\psi_j-\bar{\psi}||_{\mathbb{H}_{2}^{1,\frac{n}{2}}(K)} \to 0, \  \text{as} \ j\to \infty,\end{equation}
 that is, $$\bar{\psi}\in \mathbb{H}_{2,0}^{1,\frac{n}{2}}(\mathbb{B}_{r(\epsilon)}).$$

 By  Lemma \ref{U:lemma 6.2}, we get
 \begin{equation}\label{U:eq:61}-||u||_{L^{\infty}(\mathbb{B})}\leq u_{\epsilon}\leq ||u||_{L^{\infty}(\mathbb{B})}+\frac{r(\epsilon)}{2\epsilon}\leq 3||u||_{L^{\infty}(\mathbb{B})},\end{equation}
 and $u_{\epsilon}$ is  locally Lipschitz continuous  in the $\mathbb{B}_{r(\epsilon)}$ about $(\ln t,x)$ .
Then
 ${u}_{\epsilon}\in \mathbb{H}_{2,loc}^{1,\gamma}(\mathbb{B}_{r(\epsilon)}) $ for any $\gamma\in \mathbb{R}$. Now selecting   $\gamma=\frac{n}{2}$,  when $j\to \infty $,   by \eqref{U:eq:17}, we have
  \begin{equation}\label{U:eq:58}
  \begin{aligned}
\big|\int_{\mathbb{B}}\nabla_ \mathbb{B} {u}_{\epsilon}\cdot \nabla_ \mathbb{B} (\psi_j -\bar{\psi})\frac{dt}{t}dx\big| \leq &\int_{\mathbb{B}}\big|\nabla_ \mathbb{B} {u}_{\epsilon}\cdot \nabla_ \mathbb{B} (\psi_j -\bar{\psi})\big|\frac{dt}{t}dx=\int_{K}\big|\nabla_ \mathbb{B} {u}_{\epsilon}\cdot \nabla_ \mathbb{B} (\psi_j -\bar{\psi})\big|\frac{dt}{t}dx\\ \leq &\left(\int_{K}\left|\nabla_ \mathbb{B} {u}_{\epsilon}\right|^{2}\frac{dt}{t}dx\right)^{\frac{1}{2}}  \cdot   \left(\int_{K}|\nabla_ \mathbb{B} (\psi_j -\bar{\psi})|^{2}\frac{dt}{t}dx\right)^{\frac{1}{2}}\to 0
\end{aligned}
\end{equation}
with $\text{supp}\psi_j \Subset K$, and $ \text{supp}\bar{\psi}\Subset K \Subset \mathbb{B}_{r(\epsilon)}$.

    That is, as $j\to \infty$,
    \begin{equation}\label{U:eq:37} \int_{\mathbb{B}}\nabla_ \mathbb{B} {u}_{\epsilon}\cdot \nabla_ \mathbb{B} \psi_j \frac{dt}{t}dx\to \int_{\mathbb{B}}\nabla_ \mathbb{B} {u}_{\epsilon}\cdot \nabla_ \mathbb{B} \bar{\psi}\frac{dt}{t}dx. \end{equation}
     Combining \eqref{U:eq:61} with the non-increasing nature of $h$,
we get  \begin{equation}\label{U:eq:111}
|\left(t^2f(t,x) -t^2h(u_{\epsilon})\right)_{\epsilon}|\leq h(-||u||_{L^{\infty}(\mathbb{B})})+||f||_{L^{\infty}(\mathbb{B})} .\end{equation} Similar to  \eqref{U:eq:58},
 \begin{equation}\label{U:eq:36}  \begin{aligned}
    &\int_{\mathbb{B}} \left(-\left(t^2f(t,x) -t^2h(u_{\epsilon})\right)_{\epsilon}+(n-2)t\partial_t {u}_{\epsilon}\right)\psi_j\frac{dt}{t}dx \\ &\to \int_{\mathbb{B}} \left(-\left(t^2f(t,x) -t^2h(u_{\epsilon})\right)_{\epsilon}+(n-2)t\partial_t {u}_{\epsilon}\right)\bar{\psi}\frac{dt}{t}dx.
    \end{aligned} \end{equation}

 For any $\psi_j$, \eqref{U:1.99} holds. By \eqref{U:eq:37} and \eqref{U:eq:36}, we obtain that \eqref{U:1.99} still  holds for $\bar{\psi}$.  After simplification, we have
\begin{equation}\label{U:2.66}
\begin{aligned}
&\int_{\mathbb{B}}{\left(t(1-t)d(x,\partial X)\right)}^2 \nabla_ \mathbb{B} {u}_{\epsilon} \cdot \nabla_ \mathbb{B}\psi \frac{dt}{t}dx \\ \geq&\int_{\mathbb{B}} {\left(t(1-t)d(x,\partial X)\right)}^2 \left(-\left(t^2f(t,x) -t^2h(u_{\epsilon})\right)_{\epsilon}+(n-2)t\partial_t {u}_{\epsilon}\right)\psi\frac{dt}{t}dx\\-&
\int_{\mathbb{B}} \psi\nabla_ \mathbb{B} {u}_{\epsilon} \cdot \nabla_ \mathbb{B}{\left(t(1-t)d(x,\partial X)\right)}^2 \frac{dt}{t}dx
\end{aligned}
\end{equation} for any non-negative $\psi\in \mathcal{C}_{0}^{\infty}(\mathbb{B}_{r(\epsilon)}) $.

Step 2: We  prove that $|\nabla_ \mathbb{B} {u}_{\epsilon}|$ is uniformly bounded in $\mathbb{L}_{2}^{\frac{n-2}{2}}(\mathbb{B}_{3r(\epsilon)}\cap\{t>3r(\epsilon)\})$.

We define a cutoff function $\xi $ as follows:
 \begin{equation}\label{U:eq:60}
   \xi(t,x)= \begin{cases}
1 \ \ \ & (t,x)\in \mathbb{B}_{3r(\epsilon)}\cap\{t>3r(\epsilon)\},\\
0  \ & (t,x)\notin  \mathbb{B}_{2r(\epsilon)}\cap\{t>2r(\epsilon)\}.
\end{cases}
\end{equation}
 Set $\tilde{\psi}=(\mathop{\sup }\limits_{\mathbb{B}}{u}_{\epsilon}-{u}_{\epsilon} ){\xi}^{2}$.
 By step 1 and ${u}_{\epsilon}\in \mathbb{H}_{2,loc}^{1,\gamma}(\mathbb{B}_{r(\epsilon)})$ for any $\gamma\in \mathbb{R}$, we easily see $\tilde{\psi}\in \mathbb{H}_{2,0}^{1,\frac{n-2}{2}}(\mathbb{B}_{r(\epsilon)})$,  and analogous to \eqref{U:eq:58}, we get
\begin{equation}\label{U:2.7}
\begin{aligned}
&\int_{\mathbb{B}}{\left(t(1-t)d(x,\partial X)\right)}^2 \nabla_ \mathbb{B} {u}_{\epsilon} \cdot \nabla_ \mathbb{B}\tilde{\psi} \frac{dt}{t}dx \\ \geq&\int_{\mathbb{B}} {\left(t(1-t)d(x,\partial X)\right)}^2 \left(-\left(t^2f(t,x) -t^2h(u_{\epsilon})\right)_{\epsilon}+(n-2)t\partial_t {u}_{\epsilon}\right)\tilde{\psi}\frac{dt}{t}dx\\-&
\int_{\mathbb{B}}\tilde{\psi}\nabla_ \mathbb{B} {u}_{\epsilon} \cdot \nabla_ \mathbb{B}{\left(t(1-t)d(x,\partial X)\right)}^2 \frac{dt}{t}dx.
\end{aligned}
\end{equation}
Therefore, after simplification, we have
$$
\begin{aligned}
&\int_{\mathbb{B}}{\left(t(1-t)d(x,\partial X)\right)}^2 \left|\nabla_ \mathbb{B} {u}_{\epsilon}\right|^{2} {\xi}^{2}\frac{dt}{t}dx \\ \leq&-\int_{\mathbb{B}} {\left(t(1-t)d(x,\partial X)\right)}^2 \left(-\left(t^2f(t,x) -t^2h(u_{\epsilon})\right)_{\epsilon}+(n-2)t\partial_t {u}_{\epsilon}\right)\tilde{\psi}\frac{dt}{t}dx\\+ & \int_{\mathbb{B}}2{\left(t(1-t)d(x,\partial X)\right)}^2\xi(\mathop{\sup }\limits_{\mathbb{B}}{u}_{\epsilon}-{u}_{\epsilon})\nabla_ \mathbb{B} {u}_{\epsilon} \cdot \nabla_ \mathbb{B}\xi \frac{dt}{t}dx\\ +&
\int_{\mathbb{B}} 2\left(t(1-t)d(x,\partial X)\right)\tilde{\psi}\nabla_ \mathbb{B} {u}_{\epsilon} \cdot \nabla_ \mathbb{B}{\left(t(1-t)d(x,\partial X)\right)} \frac{dt}{t}dx.
\end{aligned}
$$
By Young's inequality, we obtain

$$\begin{aligned}
&\int_{\mathbb{B}}{\left(t(1-t)d(x,\partial X)\right)}^2\left|\nabla_ \mathbb{B} {u}_{\epsilon}\right|^{2} {\xi}^{2}\frac{dt}{t}dx  \\ \leq&\int_{\mathbb{B}}  {\left(t(1-t)d(x,\partial X)\right)}^2 \left(t^2f(t,x) -t^2h(u_{\epsilon})\right)_{\epsilon}(\mathop{\sup }\limits_{\mathbb{B}}{u}_{\epsilon}-{u}_{\epsilon} ){\xi}^{2}\frac{dt}{t}dx\\+&\theta \int_{\mathbb{B}}{\left(t(1-t)d(x,\partial X)\right)}^2\left|\nabla_ \mathbb{B} {u}_{\epsilon}\right|^{2} \xi^2\frac{dt}{t}dx +C_{1}(\theta) \int_{\mathbb{B}}{\left(t(1-t)d(x,\partial X)\right)}^2\xi^{2}(\mathop{\sup }\limits_{\mathbb{B}}{u}_{\epsilon}-{u}_{\epsilon})^{2} \frac{dt}{t}dx \\+& \theta \int_{\mathbb{B}}{\left(t(1-t)d(x,\partial X)\right)}^2\left|\nabla_ \mathbb{B} {u}_{\epsilon}\right|^{2} \xi^2\frac{dt}{t}dx+ C_{2}(\theta) \int_{\mathbb{B}}{\left(t(1-t)d(x,\partial X)\right)}^2(\mathop{\sup }\limits_{\mathbb{B}}{u}_{\epsilon}-{u}_{\epsilon})^{2} |\nabla_ \mathbb{B}\xi|^{2} \frac{dt}{t}dx\\ +&
\theta \int_{\mathbb{B}}{\left(t(1-t)d(x,\partial X)\right)}^2\left|\nabla_ \mathbb{B} {u}_{\epsilon}\right|^{2} \xi^2\frac{dt}{t}dx + C_3(\theta) \int_{\mathbb{B}}(\mathop{\sup }\limits_{\mathbb{B}}{u}_{\epsilon}-{u}_{\epsilon})^{2} {\xi}^2 |\nabla_ \mathbb{B}(t(1-t)d(x,\partial X))|^{2} \frac{dt}{t}dx
\\ =:&I_1+I_2+I_3+I_4+I_5+I_6+I_7.
\end{aligned}$$
 Choosing $\theta=\frac{1}{4}$, we have
 $$ \int_{\mathbb{B}}{\left(t(1-t)d(x,\partial X)\right)}^2\left|\nabla_ \mathbb{B} {u}_{\epsilon}\right|^{2} {\xi}^{2}\frac{dt}{t}dx \leq C(\theta)(I_1+I_3+I_5+I_7).  $$

 For $I_1$, $X$ is bounded and $0\leq \xi \leq1$,  by \eqref{U:eq:61} and \eqref{U:eq:111},  we get
$$I_1 \leq C||u||_{{L}^{\infty}(\mathbb{B})}\left(h(-||u||_{L^{\infty}(\mathbb{B})})+||f||_{L^{\infty}(\mathbb{B})} \right)<+\infty.$$

For $I_3$, similarly, we have $I_3\leq C ||u||^{2}_{{L}^{\infty}(\mathbb{B})}$.

 For $I_5$,  by \eqref{U:eq:60},   we have $|\nabla_{\mathbb{B}}\xi|\neq0 $ in  $\left(\mathbb{B}_{2r(\epsilon)}\cap\{t>2r(\epsilon)\}\right) \setminus   \left(\mathbb{B}_{3r(\epsilon)} \cap\{t>3r(\epsilon)\}\right)$.  And  we divide $\left(\mathbb{B}_{2r(\epsilon)}\cap\{t>2r(\epsilon)\}\right) \setminus   \left(\mathbb{B}_{3r(\epsilon)} \cap\{t>3r(\epsilon)\}\right)$  into three parts for separate discussion.

 when $2r(\epsilon) \leq t\leq 3r(\epsilon)$,  we have $$|\nabla_{\mathbb{B}}\xi|\leq \min\{\frac{C_1}{\ln 3-\ln2},\frac{C_1}{r(\epsilon)} \}\  \text{ and}\ \left(t(1-t)d(x,\partial X)\right)^2\leq C t^2;$$

 when $2r(\epsilon)\leq \ln1 -\ln t\leq 3r(\epsilon),$
    we have  $e^{-3r(\epsilon)}\leq t\leq e^{-2r(\epsilon)}$. Then $$|\nabla_{\mathbb{B}}\xi|\leq \frac{C_1}{r(\epsilon)} \  \text{ and}\  \left(t(1-t)d(x,\partial X)\right)^2\leq Ct(1-e^{-3r(\epsilon)})^2\leq Ct\cdot (3r(\epsilon))^2,$$
 Indeed, for the function $\omega(s)=1-e^{-3s}-3s,$  when $s\geq 0$,  $\omega'(s)=3(e^{-3s}-1)\leq 0$. Then for any $s\geq 0$, we have  $\omega(s)\leq 0.$ So  $(1-e^{-3r(\epsilon)})^2\leq (3r(\epsilon))^2$.

 when $3r(\epsilon)\leq t\leq e^{-3r(\epsilon)}$, we have
 $$|\nabla_{\mathbb{B}}\xi|\leq \frac{C_1}{r(\epsilon)} \  \text{ and}\  \left(t(1-t)d(x,\partial X)\right)^2\leq Ct(d(x,\partial X))^2\leq C t\cdot (3r(\epsilon))^2.$$
 Thus
$$\begin{aligned}I_5 &\leq C ||u||_{L^{\infty}(\mathbb{B})}^2 \left(\int_{\mathbb{B}\cap \{t\leq 3r(\epsilon)\}}\frac{t^2}{(3r(\epsilon))^2} \frac{dt}{t}dx+ \int_{\mathbb{B}}t \frac{dt}{t}dx  \right)\\&\leq C ||u||_{L^{\infty}(\mathbb{B})}^2 \left(\int_{X}\int_{0}^{3r(\epsilon)}\frac{t}{(3r(\epsilon))^2} dtdx+ \int_{\mathbb{B}}t \frac{dt}{t}dx  \right)\leq {C}||u||_{L^{\infty}(\mathbb{B})}^2.
\end{aligned}$$

For $I_7$, since $|\nabla d(x,\partial X)|=1$, we have  $ |\nabla_{\mathbb{B}} (t(1-t)d(x,\partial X))|\leq Ct$. So we obtain
$I_7\leq C  ||u||_{L^{\infty}(\mathbb{B})}^2. $

To sum up, we get
$$\int_{\mathbb{B}}{\left(t(1-t)d(x,\partial X)\right)}^2\left|\nabla_ \mathbb{B} {u}_{\epsilon}\right|^{2}  {\xi}^{2}\frac{dt}{t}dx \leq C .$$ Because $\xi=1$   on $\mathbb{B}_{3r(\epsilon)}\cap\{t>3r(\epsilon)\}$ by \eqref{U:eq:60},  we get\begin{equation}\label{U:eq:63}
 \int_{\mathbb{B}_{3r(\epsilon)}\cap\{t>3r(\epsilon)\}}{\left(t(1-t)d(x,\partial X)\right)}^2\left|\nabla_ \mathbb{B} {u}_{\epsilon}\right|^{2} \frac{dt}{t}dx \leq C,
 \end{equation} where $C$ is independent of $\epsilon$.

 Step 3: We  prove that
 $
||u||_{ \mathbb{H}_{2}^{1,\frac{n-2}{2}}(\mathbb{B})}\leq C,$
and  that $u$ is a weak supersolution of $\eqref{U:eq:11}$.

Firstly, we claim that for any $v\in \mathbb{L}_{2}^{\frac{n+2}{2}}(\mathbb{B}) $, $u\in \mathbb{L}_{2}^{\frac{n-2}{2}}(\mathbb{B}) $, $v(u)$ defined as follows is a bounded linear functional of $u$,
$$v(u):= \int_{\mathbb{B}} vu\frac{dt}{t}dx.$$
In fact, by  Young's inequality, we get $$\begin{aligned}
 |v(u)|\leq &\int_{\mathbb{B}} \left|{\left(t(1-t)d(x,\partial X)\right)} {\left(t(1-t)d(x,\partial X)\right)}^{-1}u v\right|\frac{dt}{t}dx
 \\ \leq &\left(\int_{\mathbb{B}} \left|{\left(t(1-t)d(x,\partial X)\right)}^{\frac{n}{2}-\frac{n-2}{2}} u \right|^2\frac{dt}{t}dx\right)^{\frac{1}{2}} \cdot \left(\int_{\mathbb{B}} \left|{\left(t(1-t)d(x,\partial X)\right)}^{\frac{n}{2}-\frac{n+2}{2}} v\right|^2\frac{dt}{t}dx\right)^{\frac{1}{2}}\\ =&
 ||u||_{ \mathbb{L}_{2}^{\frac{n-2}{2}}(\mathbb{B})}||v||_{ \mathbb{L}_{2}^{\frac{n+2}{2}}(\mathbb{B})}
 \end{aligned}
$$
and thus the claim is valid.

 By \eqref{U:eq:63}, we get that  $|\nabla_ \mathbb{B} {u}_{\epsilon}|$ is uniformly bounded in $\mathbb{L}_{2}^{\frac{n-2}{2}}(\mathbb{B}_{3r(\epsilon)}\cap\{t>3r(\epsilon)\})$.  Define
\begin{equation}
    \widetilde{\nabla_ \mathbb{B} {u}_{\epsilon}} =\begin{cases}
\nabla_ \mathbb{B} {u}_{\epsilon} \ \ \ & (t,x)\in \mathbb{B}_{3r(\epsilon)}\cap\{t>3r(\epsilon),\\
0  \ & (t,x)\in \mathbb B\backslash\left(\mathbb{B}_{3r(\epsilon)}\cap\{t>3r(\epsilon)\right).
\end{cases}
\end{equation}
Then   $  ||\widetilde{\nabla_ \mathbb{B} {u}_{\epsilon}}||_{\mathbb{L}_{2}^{\frac{n-2}{2}}(\mathbb{B})}\leq ||\nabla_ \mathbb{B} {u}_{\epsilon}||_{\mathbb{L}_{2}^{\frac{n-2}{2}}(\mathbb{B}_{3r(\epsilon)}\cap\{t>3r(\epsilon)\})}$.
  It implies that there is a subsequence $\widetilde{\nabla_ \mathbb{B} {u}_{\epsilon_{k}}}$ such that
  \begin{equation}\label{U:eq:64}(\widetilde{\nabla_ \mathbb{B} {u}_{\epsilon_{k}}})_{i}\to g_i(t,x),\ \ \ i=1,2,\cdots, n \end{equation}  weakly in $\mathbb{L}_{2}^{\frac{n-2}{2}}(\mathbb{B})$.
  For all $\psi\in\mathcal{C}_{0}^{\infty}(\mathbb{B})$, when $\epsilon_{k}$ is small enough,  $\text{supp} {\psi}\Subset \mathbb{B}_{3r(\epsilon)}\cap\{t>3r(\epsilon)\}$.
  Since ${u}_{\epsilon}$ converges pointwise to $u$, for all $\psi\in\mathcal{C}_{0}^{\infty}(\mathbb{B})$, we have
$$ \begin{aligned}
\int_{\mathbb{B}}g_i(t,x){\psi}\frac{dt}{t}dx=&\lim _{\epsilon_{k}\rightarrow 0}\int_{\mathbb{B}}(\widetilde{\nabla_ \mathbb{B} {u}_{\epsilon_{k}}} )_i \psi\frac{dt}{t}dx
=\lim _{\epsilon_{k}\rightarrow 0}\int_{\mathbb{B}}(\nabla_ \mathbb{B} {u}_{\epsilon_{k}} )_i \psi\frac{dt}{t}dx\\=& -\lim _{\epsilon_{k}\rightarrow 0}\int_{\mathbb{B}}{u}_{\epsilon_{k}} (\nabla_ \mathbb{B} \psi)_i\frac{dt}{t}dx=-\int_{\mathbb{B}} {u} (\nabla_ \mathbb{B} \psi)_i\frac{dt}{t}dx.
\end{aligned}
$$
 Hence we derive $\nabla_ \mathbb{B} u=g(t,x)=\left(g_1(t,x),\cdots, g_n(t,x)\right) $, which yields
\begin{equation}\label{U:eq:62}
||u||_{ \mathbb{H}_{2}^{1,\frac{n-2}{2}}(\mathbb{B})}< +\infty,\end{equation}

By \eqref{U:eq:88},  \eqref{U:1.99}  and \eqref{U:eq:64}, for any non-negative $\psi\in \mathcal{C}_{0}^{\infty}(\mathbb{B}),$
the pointwise convergence of ${u}_{\epsilon_{k}}$ to $u$,  together with the continuity of $f$ and $h$, gives that
$$ \begin{aligned} \int_{\mathbb{B}}\nabla_ \mathbb{B} {u}\cdot \nabla_ \mathbb{B}\psi \frac{dt}{t}dx &=  \lim _{\epsilon_{k}\rightarrow 0}\int_{\mathbb{B}}\widetilde{\nabla_ \mathbb{B} {u}_{\epsilon_{k}}}\cdot \nabla_ \mathbb{B}\psi \frac{dt}{t}dx= \lim _{\epsilon_{k}\rightarrow 0}\int_{\mathbb{B}}\nabla_ \mathbb{B} {u}_{\epsilon_{k}}\cdot \nabla_ \mathbb{B}\psi \frac{dt}{t}dx \\ &\geq  \lim _{\epsilon_{k}\rightarrow 0} \int_{\mathbb{B}} \left(-\left(t^2f(t,x) -t^2h(u_{\epsilon_{k}})\right)_{\epsilon_{k}}+(n-2)(t\partial_t {u}_{\epsilon_{k}})\right)\psi\frac{dt}{t}dx \\& =  \int_{\mathbb{B}} \left({t}^2h(u)- t^2f(t,x)+(n-2)(t\partial_t {u})\right)\psi\frac{dt}{t}dx.\end{aligned}
$$
 which yields that  $u$ is a weak supersolution for $\eqref{U:eq:11}$ under the definition in Lemma \ref{U:D3}.

 \end{proof}

\begin{proof}[\textbf{Proof of Proposition \ref{U:Theorem 1.6}}]By Proposition \ref{U:Theorem 1.5}, we see that $u$ is a weak supersolution to \eqref{U:eq:11}.
  Since $u$ is also a viscosity subsolution to \eqref{U:eq:11},  $-u$ is a viscosity supersolution to \eqref{U:eq:11} with $-f$ instead of $f$ and $-h(-\cdot)$ instead of $h(\cdot)$
  .  Proposition \ref{U:Theorem 1.5} implies that  $-u$ is  a weak supersolution to \eqref{U:eq:11} with $-f$ instead of $f$ and $-h(-\cdot)$ instead of $h(\cdot)$. This implies that $u$ is also a weak subsolution to \eqref{U:eq:11}.
 \end{proof}

 \begin{proof}[\textbf{Proof of Theorem \ref{U:Theorem1.7}}]
 By Theorem \ref{U:B}, when $f \in C(\overline{\mathbb{B}})\cap L^{\infty}(\mathbb{B})$, there exists  a  viscosity solution $u$ for \eqref{U:H8}, and $u$ satisfies \eqref{U:T1.10}.   Then  by  Proposition \ref{U:Theorem 1.6}, we get that $u$ is also a weak solution of \eqref{U:H8} and $u\in\mathbb{H}_{2}^{1,\frac{n-2}{2}}(\mathbb{B})$.
\end{proof}
Next, we prove the Theorem \ref{U:Theorem1.9}, which states that any continuous weak solution of \eqref{U:eq:11} is also a viscosity solution of \eqref{U:eq:11}.

\begin{lemma}\label{lemma3.2}
Let $\Omega$ be an open subset of $\mathbb{B}$,
 if $ u_i \to u $ in $ \mathbb{L}_p^\gamma(\Omega) $, then there is a subsequence $ \{u_{i_j}\}_{j=1}^\infty $ that converges to $ u $ almost everywhere on $\Omega $.
\end{lemma}
\begin{proof}
By the Definition
 \ref{U:def-1}, we have
 $$\big(\int_{\Omega}
{\left(t(1-t)d(x,\partial X)\right)}^{n-p\gamma}|u_i(t,x)-u(t,x)|^p \frac{dt}{t}dx\big)^{\frac{1}{p}}\to 0 ,$$ that is
$$
{\left(t(1-t)d(x,\partial X)\right)}^{\frac{n}{p}-\gamma}t^{-\frac{1}{p}}u_i(t,x)\to{\left(t(1-t)d(x,\partial X)\right)}^{\frac{n}{p}-\gamma}t^{-\frac{1}{p}}u(t,x) $$ in $L^{p}(\Omega)$,
then here is a subsequence $ \{{{\left(t(1-t)d(x,\partial X)\right)}^{\frac{n}{p}-\gamma}t^{-\frac{1}{p}}u_{i_j}}\}_{j=1}^\infty $ that converges to ${\left(t(1-t)d(x,\partial X)\right)}^{\frac{n}{p}-\gamma}t^{-\frac{1}{p}}u(t,x)$ almost everywhere on $\Omega $. Then $\{u_{i_j}\}_{j=1}^\infty $converges to $u(t,x)$ almost everywhere on $\Omega $.

\end{proof}

\begin{lemma}[Comparison Principle]\label{U:lemma6.6}
Let $\Omega \Subset\mathbb{B}$, $n\geq 3$, and $u_1,u_2\in \mathbb{H}^{1,\gamma}_2(\Omega)\cap L^{\infty}(\Omega)$ are respectively weak subsolution and weak supersolution  of \eqref{U:eq:11}. If $u_1\leq u_2$ on $\partial \Omega$,  then $u_1\leq u_2$ in $\Omega$. Here $u_1\leq u_2$ on $\partial \Omega$  is  $(u_1- u_2)^{+}\in \mathbb{H}^{1,\gamma}_{2,0}(\Omega).$
\end{lemma}
\begin{proof}Define $M=||(u_1-u_2)^{+}||_{L^{\infty}(\Omega)}$. If $M=0$, we get $u_1-u_2\leq 0$ in $\Omega$. Next, we will discuss the case where $M>0$ and prove that this case does not hold.
Sine $u_1,u_2\in \mathbb{H}^{1,\gamma}_2(\Omega)\cap L^{\infty}(\Omega)$ are respectively weak subsolution and weak supersolution   of \eqref{U:eq:11}, we have
\begin{equation}\label{U:eq:18}\int \limits_{\Omega} \nabla_{\mathbb{B}} u_1\cdot \nabla_{\mathbb{B}}\psi \frac{dt}{t}dx\leq  \int \limits_{\Omega} (n-2)\left(t\partial_t u_1\right) \psi+ t^2{h}(u_1)  \psi-t^2f(t,x) \psi \frac{dt}{t}dx, \end{equation}
\begin{equation}\label{U:eq:19}\int \limits_{\Omega} \nabla_{\mathbb{B}} u_2\cdot \nabla_{\mathbb{B}}\psi \frac{dt}{t}dx\geq  \int \limits_{\Omega} (n-2)\left(t\partial_t u_2\right) \psi+ t^2{h}(u_2)  \psi-t^2f(t,x) \psi \frac{dt}{t}dx\end{equation}
for all non-negative  $\psi \in C^{\infty}_{0}(\Omega)$. For $\varepsilon\leq M$,  define
\begin{equation}\label{U:eq:12}u_\varepsilon=u_1-M+\varepsilon, \ \ \ \  \omega_{\varepsilon}=u_\varepsilon-u_2,  \ \  \omega_{\varepsilon}^{+}=\max\{\omega_{\varepsilon},0\}.\end{equation}
Subtracting the two inequalities \eqref{U:eq:18} and   \eqref{U:eq:19}  yields
$$
\int \limits_{\Omega} (\nabla_{\mathbb{B}} u_1-\nabla_{\mathbb{B}} u_2)\cdot \nabla_{\mathbb{B}} \psi \frac{dt}{t}dx\leq  \int \limits_{\Omega} (n-2)\left(t\partial_t u_1-t\partial_t u_2\right) \psi+ (t^2{h}(u_{1})-t^2{h}(u_2))  \psi \frac{dt}{t}dx$$
for all non-negative  $\psi \in C^{\infty}_{0}(\Omega)$.
 Since  $(u_1- u_2)^{+}\in \mathbb{H}^{1,\gamma}_{2,0}(\Omega),$
 by  \eqref{U:eq:16}, \eqref{U:eq:17} and Lemma \ref{U:lemma5.7}, it follows that
 $(u_1- u_2)^{+}\in H^{1}_{0}(\Omega).$
By  \eqref{U:eq:12}, we have $\omega_{\varepsilon}^{+}\in H^{1}_{0}(\Omega).$
Analogously,  by  Lemma \ref{U:lemma5.7},  $\omega_{\varepsilon}^{+}\in \mathbb{H}^{1,\gamma}_{2,0}(\Omega).$  It implies that there exists a  sequence of  $\{\psi_{j}\}$ such that $ \psi_{j}\in C_{0}^{\infty}(\Omega)$  and $\psi_{j}\to \omega_{\varepsilon}^{+} $ in $\mathbb{H}^{1,\gamma}_{2,0}(\Omega)$. Then as $j\to +\infty$,
$$\begin{aligned}
&\big|\int \limits_{\Omega} (\nabla_{\mathbb{B}} u_1-\nabla_{\mathbb{B}} u_2)\cdot (\nabla_{\mathbb{B}} \omega_{\varepsilon}^{+}- \nabla_{\mathbb{B}}\psi_j )\frac{dt}{t}dx\big|\\ \leq &\left( \int \limits_{\Omega} |{\left(t(1-t)d(x,\partial X)\right)}^{\frac{n}{2}-\gamma}(\nabla_{\mathbb{B}} \omega_{\varepsilon}^{+}- \nabla_{\mathbb{B}}\psi_j)|^2 \frac{dt}{t}dx\right)^{\frac{1}{2}}\cdot \left( \int \limits_{\Omega} |{\left(t(1-t)d(x,\partial X)\right)}^{\gamma-\frac{n}{2}}\nabla_{\mathbb{B}}(u_1-u_2)|^2 \frac{dt}{t}dx\right)^{\frac{1}{2}}\\ \leq & \mathop{\sup }\limits_{\Omega}{\left(t(1-t)d(x,\partial X)\right)}^{2\gamma-n}\cdot ||u_1-u_2 ||_{\mathbb{H}^{1,\gamma}_{2,0}(\Omega)}\cdot ||\omega_{\varepsilon}^{+}-\psi_{j} ||_{\mathbb{H}^{1,\gamma}_{2,0}(\Omega)}\to 0
\end{aligned}$$
  Thus, when $j\to \infty $,
   \begin{equation}\label{eq:17}\int \limits_{\Omega} (\nabla_{\mathbb{B}} u_1-\nabla_{\mathbb{B}} u_2)\cdot \nabla_{\mathbb{B}}\psi_j \frac{dt}{t}dx\to \int \limits_{\Omega} (\nabla_{\mathbb{B}} u_1-\nabla_{\mathbb{B}} u_2)\cdot \nabla_{\mathbb{B}} \omega_{\varepsilon}^{+} \frac{dt}{t}dx\end{equation}
Similarly, since $u_1, u_2\in L^{\infty}(\Omega)$, and $h, f$ are continuous,  we have
\begin{equation}\label{eq:18}\int \limits_{\Omega} (n-2)\left(t\partial_t u_1-t\partial_t u_2\right) \psi_j \frac{dt}{t}dx\to \int \limits_{\Omega} (n-2)\left(t\partial_t u_1-t\partial_t u_2\right) \omega_{\varepsilon}^{+} \frac{dt}{t}dx,\end{equation}
\begin{equation}\label{eq:19}\int \limits_{\Omega}  t^2{h}(u_1)  \psi_{j}-t^2f(t,x) \psi_{j} \frac{dt}{t}dx\to \int \limits_{\Omega}  t^2{h}(u_1)  \omega_{\varepsilon}^{+}-t^2f(t,x) \omega_{\varepsilon}^{+}\frac{dt}{t}dx.\end{equation}
By \eqref{eq:17}, \eqref{eq:18} and  \eqref{eq:19}, we  obtain
$$ \int \limits_{\Omega} |\nabla_{\mathbb{B}} \omega_{\varepsilon}^{+}|^{2} \frac{dt}{t}dx\leq  \int \limits_{\Omega} (n-2)\left(t\partial_t \omega_\varepsilon\right)  \omega_\varepsilon^{+}+ (t^2{h}(u_{1})-t^2{h}(u_2))\omega_\varepsilon^{+} \frac{dt}{t}dx.$$
When $\omega_\varepsilon^{+}> 0$, $u_{1}\geq u_2$. And then ${h}(u_{1})\leq {h}(u_2)$.  So we have $$ \int \limits_{\Omega} |\nabla_{\mathbb{B}} \omega_{\varepsilon}^{+}|^{2} \frac{dt}{t}dx\leq  \int \limits_{\Omega} (n-2)\left(t\partial_t \omega_\varepsilon\right)  \omega_\varepsilon^{+}  \frac{dt}{t}dx.$$
By Young's inequality, we obtain
 \begin{equation}\label{eq:20}\int \limits_{\Omega} |\nabla_{\mathbb{B}} \omega_{\varepsilon}^{+}|^{2}\frac{dt}{t}dx\leq   C \int \limits_{\Omega\cap\{|\nabla_{\mathbb{B}}\omega_{\varepsilon}^{+}|>0\}}  |\omega_\varepsilon^{+}|^{2} \frac{dt}{t}dx.\end{equation}
Since $\omega_{\varepsilon}^{+}\in H^{1}_{0}(\Omega)$ and $|\nabla_{\mathbb{B}}\omega_{\varepsilon}^{+}|\geq t |\nabla \omega_{\varepsilon}^{+}|$, by applying Sobolev inequality and H\"older's inequality,  we derive
\begin{equation}\label{U:eq:81}
\begin{aligned}
&\left(\int \limits_{\Omega} | \omega_{\varepsilon}^{+}|^{\frac{2n}{n-2}} \frac{dt}{t}dx \right)^{\frac{n-2}{n}}\leq \sup\limits_{\Omega} (t^{-{\frac{n-2}{n}}}) \left(\int \limits_{\Omega} | \omega_{\varepsilon}^{+}|^{\frac{2n}{n-2}} dtdx \right)^{\frac{n-2}{n}}\leq C  \int \limits_{\Omega} |\nabla \omega_{\varepsilon}^{+}|^{2} dtdx \\ \leq & C  \sup\limits_{\Omega} (t^{-1}) \int \limits_{\Omega} |\nabla_{\mathbb{B}} \omega_{\varepsilon}^{+}|^{2} \frac{dt}{t}dx
\leq  C \int \limits_{\Omega\cap\{|\nabla_{\mathbb{B}}\omega_{\varepsilon}^{+}|>0\}}  |\omega_\varepsilon^{+}|^{2} \frac{dt}{t}dx \\ \leq & C \left(\int \limits_{\Omega\cap\{|\nabla_{\mathbb{B}}\omega_{\varepsilon}^{+}|>0\}}  1 \frac{dt}{t}dx\right)^{\frac{2}{n}}\cdot\left( \int \limits_{\Omega}  |\omega_\varepsilon^{+}|^{\frac{2n}{n-2}} \frac{dt}{t}dx\right)^{\frac{n-2}{n}}.
\end{aligned}\end{equation}
Through the  construction of $\omega_{\varepsilon}^{+}$, see \eqref{U:eq:12}, we can obtain $|\omega_{\varepsilon}^{+}|\leq \varepsilon$.
 By \eqref{eq:20}, we have
$$\int \limits_{\Omega} |\nabla_{\mathbb{B}} \omega_{\varepsilon}^{+}|^{2}\frac{dt}{t}dx\to 0, \ \text{as} \ \varepsilon\to 0.$$
Lemma \ref{lemma3.2}  implies that there exists a subsequence of  $\nabla_{\mathbb{B}} \omega_\varepsilon^{+}$, which we still denote itself, converging  to 0  a.e. in $\Omega$.
Since $$\nabla_{\mathbb{B}} \omega_\varepsilon^{+}
    =\begin{cases}
\nabla_{\mathbb{B}}(u_1-u_2)\ \ \ & \text{when} \ \omega_\varepsilon^{+}=\omega_\varepsilon \ \text{and} \ \nabla_{\mathbb{B}}(u_1-u_2)\neq 0 ,\\
0  \ & other,
\end{cases}
$$
as $\varepsilon\to 0$,  if  $\displaystyle \int \limits_{\Omega\cap\{|\nabla_{\mathbb{B}}\omega_{\varepsilon}^{+}|>0\}} 1 \frac{dt}{t}dx \nrightarrow 0$,  this contradicts that $\nabla_{\mathbb{B}} \omega_\varepsilon^{+}$ converges  to 0  a.e. in $\Omega$. So  there exists a $\varepsilon_0$ such that for any $\varepsilon\leq\varepsilon_0 $,  $$\displaystyle C\int \limits_{\Omega\cap\{|\nabla_{\mathbb{B}}\omega_{\varepsilon}^{+}|>0\}} 1 \frac{dt}{t}dx \leq \frac{1}{2}.$$
 Choosing $\varepsilon=\varepsilon_0$, by \eqref{U:eq:81},
 we  deduce that
 $$ \int \limits_{\Omega} | \omega_{\varepsilon_{0}}^{+}|^{\frac{2n}{n-2}} \frac{dt}{t}dx =0 .$$
 Then $\omega_{\varepsilon_{0}}^{+}=0$, that is, $u_1-u_2\leq M-\varepsilon_{0}$  a.e. in $\Omega$.  This property contracts with the definition of $M$. Then $M=0$ and the proof is compete.

\end{proof}

\begin{proof}[\textbf{Proof of Proposition \ref{U:Theorem1.8}}]
Let $u\in \mathbb{H}_{2}^{1,\gamma}(\mathbb{B})\cap C(\mathbb{B})$ is  a weak supersolution of
     \eqref{U:eq:11}, if $u$ is not viscosity supersolution, then there exists  $(t_0,x_0)$ and $\phi\in C^2(\mathbb{B})$ such that $u-\phi$ reach a local minimum at $(t_0,x_0)$, and
     $$\left(tr(\nabla^2_\mathbb{B}\phi)
     +(n-2)(t\partial_t \phi)\right)\big|_{(t_0,x_0)}+ t_0^2h(u(t_0,x_0))>t_0^2f(t_0,x_0) .$$
     We can choose $\varphi=\phi+(u-\phi)\big|_{(t_0,x_0)}-|(t,x)-(t_0,x_0)|^{4}_{\mathbb{B}}$ such that $u-\varphi$ reach a strict minimum value of 0 at $(t_0,x_0)$, and
     $$\left(tr(\nabla^2_\mathbb{B}\varphi)
     +(n-2)(t\partial_t \varphi)\right)\big|_{(t_0,x_0)} + t_0^2h(u(t_0,x_0))>t_0^2f(t_0,x_0) .$$
     Then there exists   a neighborhood $U$ of $(t_0,x_0)$ with  a  smooth boundary,
     such that $U\Subset  \mathbb{B}$, and  for any $(t,x)\in U$
       $$ tr(\nabla^2_\mathbb{B}\varphi)
     +(n-2)(t\partial_t \varphi) + t^2h(u(t,x))\geq t^2f(t,x) ,$$ and
   $m=\inf\limits_{\partial U} (u-\varphi)>0$. Let $\tilde{\varphi}=\varphi+m$.  Then we have
      $$tr(\nabla^2_\mathbb{B}\tilde{\varphi})
     +(n-2)(t\partial_t \tilde{\varphi}) + t^2h(u(t,x))\geq t^2f(t,x) $$ in $U$, which  means that $\tilde{\varphi}$ is a  weak subsolution of \eqref{U:eq:11}.   Since  $(\tilde{\varphi}-u)\in \mathbb{H}_{2}^{1,\gamma}(U)$, by Lemma \ref{U:lemma5.7},  we have  $(\tilde{\varphi}-u)\in H^{1}(U)$.
     Since $\tilde{\varphi}-u\leq 0$ on $\partial U$, we have $(\tilde{\varphi}-u)^{+}\in H^{1}_{0}(U).
      $ By Lemma \ref{U:lemma5.7},  $(\tilde{\varphi}-u)^{+}\in \mathbb{H}_{2,0}^{1,\gamma}(U).$   According  Lemma \ref{U:lemma6.6} we have
    $\tilde{\varphi}-u\leq 0$ in $U$, which is contradictory to $(\tilde{\varphi}-u)(t_0,x_0)=(\varphi-u)(t_0,x_0)+m=+m>0$.
 \end{proof}

\begin{proof}[\textbf{Proof of Theorem \ref{U:Theorem1.9}}]
By Proposition  \ref{U:Theorem1.8}, we have that $u$ is a viscosity supersolution of  \eqref{U:eq:11}. Since $u$ is also a weak  subsolution to \eqref{U:eq:11}, $-u$ is a weak supersolution to \eqref{U:eq:11} with $-f$ instead of $f$ and $-h(-\cdot)$ instead of $h(\cdot)$ .
 Proposition \ref{U:Theorem1.8} implies that $-u$ is  a viscosity  supersolution to \eqref{U:eq:11} with $-f$ instead of $f$ and $-h(-\cdot)$ instead of $h(\cdot)$. This implies that $u$ is also a viscosity subsolution to \eqref{U:eq:11}.
\end{proof}

\subsection*{Acknowledgments}

The authors are grateful to the referees for their careful reading and valuable comments.

\end{document}